\documentclass[11pt,a4paper]{article}
\usepackage[margin=1in]{geometry}
\usepackage{amsmath,amssymb,amsthm}
\usepackage[bookmarksnumbered=true,hidelinks]{hyperref}
\usepackage{enumitem}
\usepackage{xcolor}
\usepackage{tikz}
\usepackage{longtable}
\usepackage{array}

\newtheorem{theorem}{Theorem}
\newtheorem{lemma}[theorem]{Lemma}
\newtheorem{proposition}[theorem]{Proposition}
\newtheorem{conjecture}[theorem]{Conjecture}

\theoremstyle{definition}
\newtheorem{definition}[theorem]{Definition}
\newtheorem{remark}[theorem]{Remark}

\newcommand{\sig}[1]{\ensuremath{(#1)}}
\newcommand{\swap}{\sim}

\newif\ifstrictforwards
\ifdefined\strictforwardsflag
  \strictforwardstrue
\else
  \strictforwardsfalse
\fi

\newcounter{forwardrefcount}
\makeatletter
\newcommand{\forwardref}[1]{%
  \@ifundefined{r@#1}%
    {\ifstrictforwards
       \@latex@error{Unresolved forward reference: #1}{%
         A \string\forwardref\space target is undefined and strict mode is on.}%
     \else
       \stepcounter{forwardrefcount}%
       \textcolor{orange}{\textbf{[forward: #1]}}%
       \@latex@warning{Forward reference to undefined label '#1'}%
     \fi}%
    {\ref{#1}}%
}
\makeatother

\newcommand{\reportforwardrefs}{%
  \ifnum\value{forwardrefcount}>0
    \typeout{========================================}%
    \typeout{UNRESOLVED FORWARD REFERENCES: \arabic{forwardrefcount}}%
    \typeout{(Search the log for 'Forward reference to undefined label')}%
    \typeout{========================================}%
  \fi
}

\AtEndDocument{\reportforwardrefs}

\IfFileExists{ornamental-braid.pdf}{%
  \ifdefined\pdfximage
    \usepackage{graphicx}%
    \usepackage{eso-pic}%
    \pdfximage{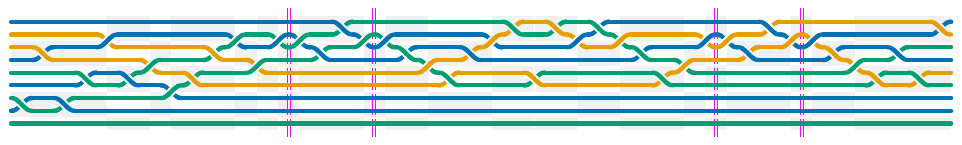}%
    \edef\decoMaxPage{\the\pdflastximagepages}%
    \newcounter{decoPage}%
    \AddToShipoutPictureBG{%
      \stepcounter{decoPage}%
      \ifnum\value{decoPage}>\decoMaxPage\relax\else
        \AtPageLowerLeft{%
          \includegraphics[page=\thedecoPage]{ornamental-braid.pdf}}%
      \fi}%
  \fi
}{}%

\title{A proof of\\
  Lehmer's permutation conjecture for neighbor-swap graphs}
\author{Tom Verhoeff\\
  Department of Mathematics and Computer Science\\
  Eindhoven University of Technology (TU/e)\\
  Eindhoven, The Netherlands\\
  \href{mailto:T.Verhoeff@tue.nl}{T.Verhoeff@tue.nl}}

\date{September 30, 2026}

\begin{document}
\maketitle

\begin{abstract}
In 1965, D.~H.~Lehmer conjectured that the permutations of every
multiset admit an \emph{imperfect} Hamiltonian traversal by adjacent
swaps: a walk in the neighbor-swap graph that visits every word, with some
words visited twice in order to reach a neighbor and return.  The
question is posed as an unsolved research problem in Knuth's
\emph{Art of Computer Programming}.  Verhoeff (2017) chose the stutter
words, in which every domino is a double, as the words to be reached
this way, and reformulated the conjecture as the Hamiltonicity of the
graph $N(S)$ on the non-stutter words, with two exceptional families
--- binary signatures with an odd multiplicity, and the permutations
of $(2k,1,1)$ --- that admit a Hamiltonian path but no cycle.  This
article proves the reformulated conjecture, and with it Lehmer's
conjecture.  The key structure is a partition of the words into
\emph{hypercubes}: the swaps inside dominoes turn each class of words
with the same domino contents into a hypercube, and the stutters are
exactly the $0$-dimensional classes.  When every multiplicity is even,
Hamiltonian cycles of the hypercubes are glued along a spanning tree,
with no finite check.  The case of exactly one odd multiplicity
reduces to the all-even case and to a theorem of Stachowiak (1992),
the one inherited Hamiltonicity input, which also settles two or more
odd multiplicities.  The only finite ingredients are two explicit
cycles, of $28$ and $84$ words.  Every construction is implemented in
Python and checked against brute-force graphs, and the proof is
formalized in Lean~4 over Mathlib.
\end{abstract}

\tableofcontents

\section{Introduction}
\label{sec:intro}

\subsection{Lehmer's permutation conjecture}

A word is a sequence of colored objects.  Our objects are the
nonnegative integers and the color of an object is its value; we often
say \emph{color} for a colored object.  For example,
\[
  013320122
\]
is a word of length $n = 8$ over $K = 4$ colors, each occurring twice.
The number of occurrences of a color is its multiplicity, and the
tuple of multiplicities is the \emph{signature} of the word; here it
is $S = (2,2,2,2)$.  Signatures are considered up to renaming the
colors.  Two words are \emph{neighbors} when one is obtained from the
other by transposing two adjacent, distinct colors; the
\emph{neighbor-swap graph} $G(S)$ has as its vertices the words of
signature $S$, with the neighbor pairs as edges.  (The formal
definitions are collected in Section~\ref{sec:preliminaries}.)
Lehmer \cite{lehmer1965} conjectured in 1965 that every such graph
admits an \emph{imperfect} Hamiltonian traversal: a walk visiting
every vertex, with each vertex visited exactly once except possibly
some \emph{spur bases}, visited twice in order to reach a
\emph{spur tip} (a vertex at distance $1$ from the walk) and return.
Lehmer prescribed neither the spur tips nor a minimum count; his
conjecture asks only that a walk of this form exist.  The question is
also Exercise~53 of \S7.2.1.3 of Knuth's \emph{Art of Computer
Programming} \cite{knuth2011}, rated there among the unsolved
research problems; this article solves that exercise.

Verhoeff \cite{verhoeff2017} reformulated the conjecture by
identifying the \emph{stutter words} --- those in which every
domino is a double (Section~\ref{prelim:stutter}) --- as a natural
and optimal choice of spur tips, and reduced it to a Hamiltonicity statement for
the subgraph $N(S) \subseteq G(S)$ induced by the non-stutter
words.  This reduced statement, the \emph{reformulated Lehmer
conjecture}, asserts that $N(S)$ admits a Hamiltonian cycle for every
signature $S$ of arity at least $2$ outside two families: the binary
signatures with an odd multiplicity, and the permutations of
$\sig{2k,1,1}$.  These two families are exactly the signatures whose
non-stutter graph admits a Hamiltonian path but no Hamiltonian cycle:
the binary classification is Verhoeff's
\cite[Theorem~9]{verhoeff2017} (proved earlier for the path in the
odd case by Eades, Hickey, and Read \cite{eades1984}), and the family
$\sig{2k,1,1}$ is excluded by a parity argument across an edge cut,
due to Stachowiak \cite[Lemma~12]{stachowiak1992} and recorded by
Verhoeff \cite[\S5.4]{verhoeff2017}.  At two or more odd
multiplicities, the graph has no stutters and a Hamiltonian cycle is
supplied by Stachowiak's theorem \cite{stachowiak1992}.  The open case
of the reformulated conjecture is therefore arity at least $3$ with at
most one odd multiplicity; Verhoeff \cite[Theorem~13]{verhoeff2017}
proved that a disjoint cycle cover exists there, and sketched the
family $\sig{2j,2,1}$ \cite[\S5.3]{verhoeff2017}.  Opperman
\cite{opperman2024} pursued a merge-of-covers route toward the whole
conjecture in a master's thesis.  That treatment is incomplete: some
cross edges that its gluing steps require are not among the edges it
establishes.\footnote{For example, in \S3.9.2.5 of the thesis the cross
edge~(3.126) is identified with the guaranteed edge~(3.166) of
\S3.11.  The two are different edges: with the two trailing colors
removed, the words of~(3.126) end in a color of even multiplicity in
their signature, while the words of~(3.166) end in a color of odd
multiplicity.}  No result of it is used here, and to our knowledge the
theorem is open in the refereed literature.  The
proof below relies on exactly the two inherited results stated in
Section~\ref{sec:inherited} --- Verhoeff's binary classification and
Stachowiak's theorem --- and on nothing else from the literature; in
particular the Hamiltonian cycles of the hypercubes that it uses are
constructed here (Lemma~\ref{cube:boustrophedon}).

\subsection{The result}

This article proves the reformulated conjecture, and with it Lehmer's
conjecture:

\medskip
\noindent\textbf{Main theorem} (Theorem~\ref{assembly:theorem}).
\emph{Let $S$ be a signature of arity at least $2$ that is neither a
binary signature with an odd multiplicity nor a permutation of $\sig{2k,1,1}$.  Then $N(S)$ admits a Hamiltonian cycle.}
\medskip

\noindent
At the two exceptional families a Hamiltonian path exists and no cycle
does, so the exception set is exact; Verhoeff's reduction then turns
the cycles, and the paths at the exceptions, into Lehmer paths of
$G(S)$ (Remark~\ref{assembly:lehmer}).  The proof is assembled in
Section~\ref{sec:assembly} from two new theorems, at zero and at one
odd multiplicity.  The inherited Stachowiak theorem is used at two
levels: at the top level it covers the signatures with two or more
odd multiplicities outright, and inside the one-odd proof it supplies the cycles of the
two-odd blocks, which are used as arbitrary Hamiltonian cycles, with
no property assumed beyond existence (Section~\ref{sec:one-odd}).

\subsection{The proof in outline}

Group the positions into \emph{dominoes}: the two-position
containers $(0,1)$, $(2,3)$, \ldots, where domino $i$ consists of
positions $2i$ and $2i+1$ and carries the two colors that occupy them
(when the length is odd the last position is left over), as in
$01331022 = 01|33|10|22$.  The \emph{shifted dominoes} are the same
containers shifted one position.  A domino is
\emph{mixed} if the two colors it carries differ and a \emph{double}
otherwise.  The \emph{skeleton} records the content of each domino as
an unordered pair of colors; for instance $01331022$, $10331022$,
$01330122$ and $10330122$ all have the skeleton $01|33|01|22$.  A
swap inside a domino (an \emph{intra-domino} swap) flips it,
exchanging the two colors it carries, so it stays inside the set of
words with a given skeleton.  A swap between adjacent dominoes (an \emph{inter-domino} swap)
flips a shifted domino, exchanging its two colors; it changes both
dominoes, and it changes the skeleton exactly when those two colors
differ.

The skeleton classes are the key structure.  The intra-domino swaps
turn a skeleton class into a hypercube, with one coordinate per mixed
domino, and the stutters are exactly the $0$-dimensional classes.
The class of the example is a square, and a class with three mixed
dominoes is a $3$-cube (Figure~\ref{intro:skeleton-figure}).

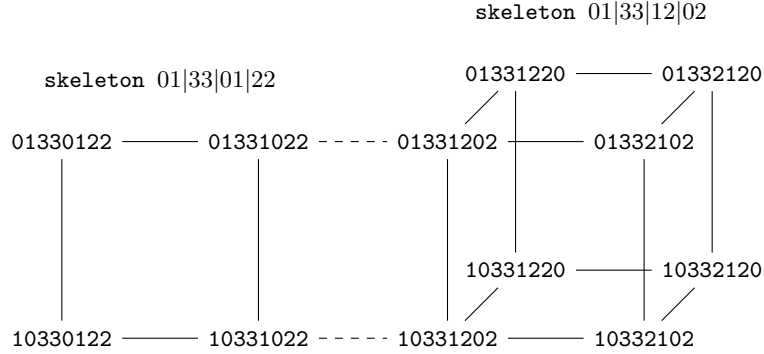
\begin{figure}[ht]
\centering
\begin{tikzpicture}[scale=1.0,
  word/.style={font=\ttfamily\footnotesize}]
  % ---- the square class (left), rotated so that its right edge faces
  % ---- the cube's front-left edge (both flip the leftmost domino)
  \node[word] (C) at (0.3,2.6) {01330122};
  \node[word] (A) at (2.9,2.6) {01331022};
  \node[word] (D) at (0.3,0.0) {10330122};
  \node[word] (B) at (2.9,0.0) {10331022};
  \draw (C)--(A);
  \draw (A)--(B);
  \draw (B)--(D);
  \draw (D)--(C);
  \node[word] at (1.6,3.4) {skeleton $01|33|01|22$};
  % ---- the 3-cube class (right): two shifted squares; vertical edges
  % ---- flip the leftmost domino, the four rungs are diagonal
  \node[word] (F00) at (5.4,2.6) {01331202};
  \node[word] (F10) at (5.4,0.0) {10331202};
  \node[word] (F01) at (8.0,2.6) {01332102};
  \node[word] (F11) at (8.0,0.0) {10332102};
  \node[word] (G00) at (6.3,3.5) {01331220};
  \node[word] (G10) at (6.3,0.9) {10331220};
  \node[word] (G01) at (8.9,3.5) {01332120};
  \node[word] (G11) at (8.9,0.9) {10332120};
  \draw (F00)--(F10);
  \draw (F10)--(F11);
  \draw (F11)--(F01);
  \draw (F01)--(F00);
  \draw (G00)--(G10);
  \draw (G10)--(G11);
  \draw (G11)--(G01);
  \draw (G01)--(G00);
  \draw (F00)--(G00);
  \draw (F10)--(G10);
  \draw (F01)--(G01);
  \draw (F11)--(G11);
  \node[word] at (7.3,4.3) {skeleton $01|33|12|02$};
  % ---- the two inter-domino edges straight across ----
  \draw[dashed] (A)--(F00);
  \draw[dashed] (B)--(F10);
\end{tikzpicture}
\caption{Two skeleton classes.  The four words with skeleton
$01|33|01|22$ (left) form a square; the eight with skeleton
$01|33|12|02$ (right) form a $3$-cube, whose four rungs --- the flips
of its rightmost mixed domino --- are the slanted edges.  The two
dashed edges are inter-domino swaps, pairing $01331022$ with
$01331202$ and $10331022$ with $10331202$; all solid edges are
intra-domino swaps.}
\label{intro:skeleton-figure}
\end{figure}

With every multiplicity even, no class has dimension $1$, so the
non-stutter graph is the disjoint union of the hypercubes of dimension
at least $2$, joined by inter-domino edges.  Every hypercube carries a
Hamiltonian cycle that sweeps back and forth through all its rungs
--- the edges that flip its rightmost mixed domino --- and these
cycles are then merged along a spanning tree of the hypercubes
(Figure~\ref{intro:stitch-figure}): each
edge of the tree is realized by a stitch carrying one hypercube edge
in each of the two hypercubes, and splicing the two cycles at these
edges merges them into one.  The result is a Hamiltonian cycle of
$N(S)$ with no finite check (Sections~\ref{sec:tools}
and~\ref{sec:all-even}).

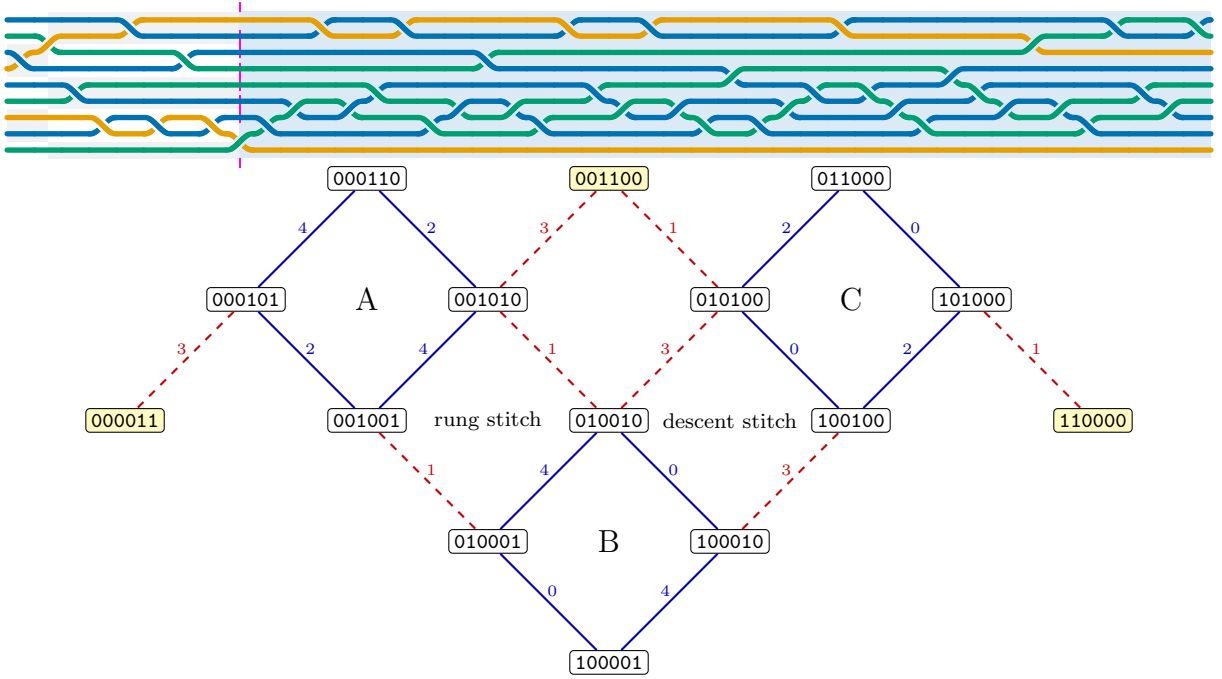
\begin{figure}[ht]
\centering
\begin{tikzpicture}[
  x=1.6cm, y=1.6cm,
  word/.style={draw, rounded corners=1.5pt, inner sep=2.0pt,
               font=\ttfamily\scriptsize, fill=white},
  stut/.style={word, fill=yellow!30},
  even/.style={thick, draw=blue!75!black},
  odd/.style={thick, dashed, draw=red!80!black},
  elabul/.style={font=\tiny, text=blue!75!black, inner sep=0.5pt,
                 anchor=south, pos=0.5, xshift=-1.5pt, yshift=1.5pt},
  elabur/.style={font=\tiny, text=blue!75!black, inner sep=0.5pt,
                 anchor=south, pos=0.5, xshift=1.5pt, yshift=1.5pt},
  olabul/.style={font=\tiny, text=red!80!black, inner sep=0.5pt,
                 anchor=south, pos=0.5, xshift=-1.5pt, yshift=1.5pt},
  olabur/.style={font=\tiny, text=red!80!black, inner sep=0.5pt,
                 anchor=south, pos=0.5, xshift=1.5pt, yshift=1.5pt},
]
  % vertices: (inversion count, layer)
  \node[stut] (S1) at (0,-1) {000011};
  \node[word] (A1) at (1, 0) {000101};
  \node[word] (A2) at (2,-1) {001001};
  \node[word] (A3) at (2, 1) {000110};
  \node[word] (A4) at (3, 0) {001010};
  \node[word] (B1) at (3,-2) {010001};
  \node[word] (B2) at (4,-1) {010010};
  \node[word] (B3) at (4,-3) {100001};
  \node[stut] (S2) at (4, 1) {001100};
  \node[word] (B4) at (5,-2) {100010};
  \node[word] (C1) at (5, 0) {010100};
  \node[word] (C2) at (6,-1) {100100};
  \node[word] (C3) at (6, 1) {011000};
  \node[word] (C4) at (7, 0) {101000};
  \node[stut] (S3) at (8,-1) {110000};
  % cube A (skeleton 00|01|01)
  \draw[even] (A1)--(A2) node[elabur] {2};
  \draw[even] (A1)--(A3) node[elabul] {4};
  \draw[even] (A2)--(A4) node[elabul] {4};
  \draw[even] (A3)--(A4) node[elabur] {2};
  \draw[even] (B1)--(B2) node[elabul] {4};
  \draw[even] (B1)--(B3) node[elabur] {0};
  \draw[even] (B2)--(B4) node[elabur] {0};
  \draw[even] (B3)--(B4) node[elabul] {4};
  \draw[even] (C1)--(C2) node[elabur] {0};
  \draw[even] (C1)--(C3) node[elabul] {2};
  \draw[even] (C2)--(C4) node[elabul] {2};
  \draw[even] (C3)--(C4) node[elabur] {0};
  \draw[odd] (A2)--(B1) node[olabur] {1};
  \draw[odd] (A4)--(B2) node[olabur] {1};
  \draw[odd] (B2)--(C1) node[olabul] {3};
  \draw[odd] (B4)--(C2) node[olabul] {3};
  \draw[odd] (A1)--(S1) node[olabul] {3};
  \draw[odd] (A4)--(S2) node[olabul] {3};
  \draw[odd] (C1)--(S2) node[olabur] {1};
  \draw[odd] (C4)--(S3) node[olabur] {1};
  % cube labels at the centres of the squares
  \node[font=\large] at (2,0) {A};
  \node[font=\large] at (4,-2) {B};
  \node[font=\large] at (6,0) {C};
  % stitch names, each inside its square
  \node[font=\scriptsize] at (3,-1) {rung stitch};
  \node[font=\scriptsize] at (5,-1) {descent stitch};

\end{tikzpicture}
\caption{The smallest all-even signature whose families have more than
one cube: the twelve non-stutter permutations of $\sig{4,2}$ form the
squares $A$, $B$, $C$ (skeletons $00|01|01$, $01|00|01$, $01|01|00$),
and the three yellow permutations are the stutters, the
$0$-dimensional cubes.  Vertices are permutations, placed by inversion
count; solid blue edges swap inside a domino, dashed red edges
swap across two dominoes, and an edge carries its swap position $s_p$.
$A$ and $B$ form the family $F(\emptyset,01)$ --- their rightmost
mixed domino is the last, with no double after it --- and $C$ the
family $F((00),01)$.  The spanning tree is built first inside the
families: the rung stitch joins the rungs $s_4$ of $A$ and $B$ by the
cross swaps $s_1$.  Then the families are joined: the descent stitch
attaches $F((00),01)$ to $F(\emptyset,01)$ by swapping $s_3$ --- the
end of $C$'s rung with its trailing double --- and using $s_0$ as its
domino edges.  The cycle visits no stutter.}
\label{intro:stitch-figure}
\end{figure}

The one-odd theorem (Section~\ref{sec:one-odd}) splits by the last
color.  Let $x$ be the color of odd multiplicity and
$S^x = S - x$; it is all even, so the words ending in $x$ ---
exactly the words of $S^x$ with $x$ appended --- are cycled by the
previous theorem.  For every other color $z$, the words ending
in $z$ are the words of $S^z$ with $z$ appended, and $S^z$
has exactly two odd multiplicities, so Stachowiak's theorem supplies
a cycle
--- except when $S^z$ is a permutation of $(2j,1,1)$, the two-odd
family with no Hamiltonian cycle.  That exception occurs exactly when
$S$ is a permutation of $(2j,2,1)$; this family is handled by its own
grid construction (Section~\ref{sec:family}).  In the general case
the blocks are glued along parallel pairs whose prescribed edge
inside the all-even part is read off its cycle.  Each other block is
entered at a \emph{shifted stutter}: an anchored word (one
with no inter-domino edge) that uses every color of $S^x$.  When
$S^x$ has at least three colors, it is therefore the only anchored
word in its hypercube (Lemma~\ref{one-odd:anchored}).  A degree argument forces one
edge of every Hamiltonian cycle at the shifted stutter, and its
partner edge can be demanded of the all-even cycle because it lies in
a hypercube of dimension at least four, which the all-even
construction glues only along rungs
(Proposition~\ref{one-odd:routing}).  The two theorems are not
mutually inductive: the one-odd theorem reduces to the all-even
theorem and to the inherited two-odd theorem, while the all-even
proof is self-contained, its only induction being inside
Lemma~\ref{tools:lemmaM}.  The only finite ingredients of the whole
proof are two explicit cycles of $28$ and $84$ words
(Appendix~\ref{appendix:certificates}).

\subsection{Provenance: how the proof was found}
\label{intro:provenance}

The proof has an unusual origin, recorded here because it illustrates
how the problem was solved.  For about half a year before it, the
project had pursued a different and much longer route to the same
case: a decomposition by trailing colors, the prescription of gluing
edges \emph{before} the block cycles were constructed, and, earlier
still, the merging of a cycle cover across class seams.  That route
required a large case analysis and several finite certificates of
considerable size.  That route is not published, and nothing of it is
used here.

In September 2026 the problem was re-examined with a deliberately
different brief: find a short and elegant proof, study the failure
patterns, and survey representations of the problem that had not been
tried.  In that study phase, Claude Opus~5.5 proposed the
representation that became the partition into hypercubes: a
word viewed as a sequence of dominoes.  It noticed that the even-position swaps
generate a group whose orbits are hypercubes, that the stutters are
exactly the fixed points, and that gluing the Gray cycles of the
orbits already produced Hamiltonian cycles on small all-even
instances.  It then developed this lead into the two theorems above,
in successive multi-agent rounds: writing, independent recomputation
of every construction, adversarial red-teaming, and hand refereeing.
That proof is the source of the proof text of
Sections~\ref{sec:all-even}--\ref{sec:family}; this article
re-derives it in the article's notation, subjects it to the
independent machine verification of
Appendix~\ref{appendix:machine-verification}, and formalizes it in
Lean~4 (Appendix~\ref{appendix:formalization}).  The contrast between
the two routes --- one idea about a representation against months of
case analysis --- is the main methodological lesson of the project.

\subsection{Relation to prior work}

The components of the proof have clear antecedents.
The parity-balancing involution that makes the stutters the parity
surplus is due to Ruskey \cite{ruskey1992} and reappears in
Stachowiak \cite{stachowiak1997} and Verhoeff
\cite[Theorem~4]{verhoeff2017}.  The gluing of two cycles along a
parallel pair is Stachowiak's device \cite[Lemma~3]{stachowiak1992},
recorded by Verhoeff \cite[Lemma~11]{verhoeff2017}; iterated along a
spanning tree it is the pattern behind
Conway--Sloane--Wilks \cite{conway1989} and, in modern form,
Gregor--M\"utze--Nummenpalo \cite{gregor2018}.  The boustrophedon
cycle of a hypercube is the product construction of
Conway--Sloane--Wilks \cite{conway1989} and Verhoeff's
Lemma~12 \cite{verhoeff2017}; the hypercubes themselves are orbits of the
parabolic subgroup generated by the even swaps (parabolic double
cosets \cite{billey2018}).  Splitting by the last domino or the last
color is Verhoeff's \cite[\S\S5.1--5.2]{verhoeff2017}, and the
arity-reduction of the two-odd case by fibers is Stachowiak's
\cite{stachowiak1992}.  The family $\sig{2j,2,1}$ was sketched by
Verhoeff \cite[\S5.3]{verhoeff2017} and attributed there to van
Duijnhoven \cite{vanduijnhoven2013}; the forced-edge insertion used
here is a small explicit completion.  As far as we know, the
following are new: the matching of odd swaps that is a spanning tree
of the hypercubes, with the free-word invariant (Lemma~\ref{tools:lemmaM}),
and the resulting all-even proof, which needs no finite certificate
or computer check (its only induction is inside
Lemma~\ref{tools:lemmaM}); the shifted-stutter
forcing of Section~\ref{sec:one-odd}, which lets arbitrary cycles of
the two-odd blocks be glued without prescribing any edge inside them;
and the two-odd certificate-free route through the inherited theorem
as it is consumed here.  Opperman's thesis \cite{opperman2024}
follows a different, case-based route, and no result of it is used.

\subsection{Machine verification and formalization}

Every construction in this article is implemented, in the same
generality as its proof, in the Python package \texttt{nswap}
distributed with this submission, and checked against brute-force
neighbor-swap graphs at small parameters:
Appendix~\ref{appendix:machine-verification} states the claims and
the exact windows.  The proof is additionally formalized in Lean~4
over Mathlib: the hypercube proof and the Stachowiak theorem are
formalized in separate packages, and the Lean development contains
the bridges that
assemble them into the main theorem, with no \texttt{sorry} and
standard axioms only (Appendix~\ref{appendix:formalization}).  The
formal development is documented statement by statement, with every
mathematical statement linked to its Lean declaration, so that it can
be ported to later toolchains and, where stable, offered to Mathlib.

\subsection{Organization}

Section~\ref{sec:preliminaries} fixes the vocabulary, Verhoeff's
reduction, the gluing lemma, and the partition into hypercubes, and
Section~\ref{sec:inherited} states the two inherited results.
Sections~\ref{sec:tools}--\ref{sec:family} carry out the proof
outlined above, and Section~\ref{sec:assembly} assembles the main
theorem.  The appendices give the two explicit cycles, the machine
verification record, and the Lean formalization.  A note at the end
explains the ornamental braid of the decorated edition.

\section{Preliminaries}
\label{sec:preliminaries}
This section fixes the vocabulary --- signatures, words, and
the neighbor-swap graph --- and develops the two structures the proof
is built on: the stutter words, with Verhoeff's reduction of
Lehmer's conjecture, and the partition into hypercubes, with the
gluing lemma used to assemble the cycles.
\subsection{Signatures, words, and the neighbor-swap graph}
\label{prelim:signatures}

A \emph{signature} is a tuple $S = (k_0, k_1, \ldots, k_{K-1})$ of
non-negative integers, with $K = K(S)$ the \emph{arity} and
$n = \sum_{c=0}^{K-1} k_c$ the \emph{length}.  A \emph{word} of
signature $S$ is a sequence $p = p_0 p_1 \cdots p_{n-1}$ of colored
objects in which color $c$ occurs $k_c$ times, for every color $c$.
As in the introduction, the colored objects are the nonnegative
integers, and we say \emph{color} for a colored object.
A color of positive multiplicity is \emph{active}, and the active
colors form the \emph{support} of $S$.  There are
$n!/\prod_c k_c!$ words.  For a color $c$ we write
$S^c = S - c$ for the signature with $k_c$ decreased by one (defined
whenever $k_c \ge 1$); more generally, $S - u$ for a word $u$ denotes
the signature with the colors of $u$ removed, and
$S - a - b - \cdots$ denotes iterated subtraction.  Such an expression
is defined only when no multiplicity becomes negative.  Renaming the
colors is an isomorphism of everything below, so statements about a
signature hold for every renaming of its colors.

\emph{Positions are numbered from~$0$}.  For $0 \le i < n-1$, let
$s_i(p)$ be the word obtained from $p$ by exchanging the colors at
positions $i$ and $i+1$; we call this a \emph{swap} (at position $i$),
and it is a valid neighbor swap exactly when the two colors differ.
The \emph{neighbor-swap graph} $G(S)$ has as its vertex set $V(S)$,
the words of~$S$, with an edge between $p$ and $s_i(p)$ for every valid
swap; we write $p \swap p'$ for such an edge.

The positions are grouped into \emph{dominoes}: the two-position
containers $(0,1)$, $(2,3)$, \ldots; domino $i$ consists of positions
$2i$ and $2i+1$ and \emph{carries} the two colors that occupy its
cells (if $n$ is odd the last position is left over).  The
\emph{shifted dominoes} are the same containers shifted one position:
shifted domino $i$ consists of positions $2i+1$ and $2i+2$.  A domino
or shifted domino is \emph{mixed} if the two colors it carries differ
and a \emph{double} otherwise.  A swap $s_i$ with even $i$
\emph{flips} the domino it lies in: it exchanges the two colors that
domino carries, and is called an \emph{intra-domino} (or \emph{even})
swap.  A swap with odd $i$ flips the shifted domino it lies in: the
two colors it exchanges belong to two different dominoes, so it
changes both, and it is called an \emph{inter-domino} (or \emph{odd})
swap.  An edge produced by $s_i$ inherits the name \emph{even} or
\emph{odd} from its swap.

The graph $G(S)$ is bipartite: the inversion count (the number of
pairs of positions whose colors are out of their natural order)
changes parity under every swap.  Its two classes are the words of
even and of odd inversion count.

\subsection{Stutter words}
\label{prelim:stutter}

A \emph{stutter}, or \emph{stutter word}, is a word in which every
domino is a double; that is, it has the form $a a b b \cdots y y$,
with a possibly trailing single color $z$ when $n$ is odd.  (In the
literature this is called a \emph{stutter permutation}.)  Let
$N(S) \subseteq G(S)$ be the subgraph induced by the non-stutter
words.  Renaming colors maps stutters to stutters, so $N(S)$ and its
statements transport under every renaming of the colors.

\begin{lemma}[Parity census; Ruskey]
\label{prelim:stutter-parity}
For every signature $S$, the graph $G(S)$ is bipartite with classes
the words of even and of odd inversion count.  Every stutter lies in
the even class, and
\[
  \#\{p : \operatorname{inv}(p) \text{ even}\}
  - \#\{p : \operatorname{inv}(p) \text{ odd}\}
  = \#\{\text{stutters of } S\}.
\]
In particular the two bipartition classes of $N(S)$ have equal size.
\end{lemma}

\begin{proof}
A swap transposes two adjacent \emph{distinct} colors and changes
the inversion count by exactly $\pm 1$; hence every edge of $G(S)$
joins the two parity classes.

A stutter has even inversion count.  Its positions fall into its
doubles and, when $n$ is odd, the trailing single.  The two positions
of one double hold equal colors and form no inversion.  For two
different doubles $cc$ and $dd$, the four pairs of positions with one
position in each are all inversions if the earlier double carries the
larger color, and none are otherwise; so the two doubles contribute
$0$ or $4$.  Likewise a double and the trailing single contribute $0$
or $2$.

For the census, define $\iota(p)$, for a non-stutter $p$, as the word
obtained by transposing the colors of the \emph{first} mixed domino.
The first mixed domino of $\iota(p)$ is the same domino (earlier
dominoes are untouched and double), so $\iota$ is an involution on the
non-stutters, and it transposes two adjacent distinct colors, so it
flips the inversion parity.  The non-stutters therefore cancel in the
signed count, leaving one $+1$ for each stutter; removing the
stutters, all in the even class, balances the classes of $N(S)$
exactly.
\end{proof}

The balance clause is a necessary condition for the cycles this
article constructs: a Hamiltonian cycle of a bipartite graph
alternates between its two classes, so those classes must be equal in
size.  The involution of the proof is due to Ruskey
\cite{ruskey1992}; it reappears in Stachowiak \cite{stachowiak1997}
and in Verhoeff \cite[Theorem~4]{verhoeff2017}, where its fixed
points --- the stutters --- are used as the spurs of Lehmer's
conjecture.

\begin{lemma}[Stutters are far apart; Verhoeff
{\rm\cite[Lemma~6(6)]{verhoeff2017}}]\label{prelim:stutter-distance}
Two distinct stutters are at distance at least $4$ in $G(S)$.  In
particular, no word is adjacent to two distinct stutters.
\end{lemma}

\begin{proof}
Label the occurrences of each color $c$ in a word from left to right
as $c_1, c_2, \ldots, c_{k_c}$.  A swap exchanges two adjacent
occurrences of distinct colors, so it keeps every label and reverses
the relative order of exactly one pair of labeled occurrences.  For
words $p$ and $q$ let $\delta(p,q)$ be the number of pairs of labeled
occurrences whose relative order differs between $p$ and $q$.  A swap
changes $\delta(\cdot, q)$ by exactly one, so the distance between $p$
and $q$ is at least $\delta(p,q)$.

Let $s \ne s'$ be stutters.  In both, the occurrences $c_{2t-1},
c_{2t}$ of each color $c$ fill one double, and when $n$ is odd the
trailing position holds the last occurrence of the odd color.  Two
occurrences of the same color, or two involving the trailing
position, have the same relative order in $s$ and $s'$.  The
occurrences of two different doubles are ordered as the doubles are,
so each pair of doubles contributes $0$ or $4$ to $\delta(s,s')$.
Since $s \ne s'$, some pair of doubles is ordered differently, and
$\delta(s,s') \ge 4$.
\end{proof}

\subsection{From Lehmer's conjecture to the non-stutter graph}
\label{prelim:restriction-section}

A \emph{Lehmer path} in $G(S)$ is a walk that visits every vertex at
least once, with each vertex visited exactly once \emph{except}
possibly some \emph{spur bases}, which are visited twice in order to
reach a \emph{spur tip} (a vertex of distance $1$ from the walk) and
return.  Lehmer's permutation conjecture asserts that every $G(S)$
admits a Lehmer path.  Verhoeff \cite{verhoeff2017} proposed the
stutters as spur tips and reduced the conjecture as follows.

\begin{theorem}[Verhoeff {\rm\cite[cf.~Thm.~7]{verhoeff2017}}]%
\label{prelim:restriction}
If $N(S)$ admits a Hamiltonian path or cycle, then $G(S)$ admits a
Lehmer path with spur tips among the stutters.
\end{theorem}

\begin{proof}[Proof sketch]
A stutter of arity at least two is adjacent in $G(S)$ to a non-stutter:
either two consecutive dominoes carry different colors, or the last
domino and the trailing color differ; flipping that shifted domino
gives a non-stutter.  Choose for every stutter one such
adjacent non-stutter and attach the stutter to it as a spur tip.  By
Lemma~\ref{prelim:stutter-distance} no non-stutter is adjacent to two
stutters, so each chosen non-stutter carries exactly one spur.
Traversing the Hamiltonian path or cycle of $N(S)$ and making an
excursion to each attached stutter therefore visits every stutter
exactly once and every spur base exactly twice.
\end{proof}

So it suffices to establish Hamiltonicity of $N(S)$: a cycle outside
the exceptional families, and a path at the two families whose
non-stutter graphs admit no cycle
(Section~\ref{prelim:trichotomy-section}).

\subsection{The reformulated conjecture, and its two exceptions}
\label{prelim:trichotomy-section}

Verhoeff \cite{verhoeff2017} reformulated Lehmer's conjecture as a
Hamiltonicity statement for the non-stutter graph and identified the
two families in which no Hamiltonian cycle can exist.

\begin{conjecture}[Reformulated Lehmer conjecture; Verhoeff
{\rm\cite[Conjecture~8]{verhoeff2017}}]\label{prelim:trichotomy}
Let $S$ be a signature of arity at least $2$ that is neither a binary
signature with an odd multiplicity nor a permutation of $\sig{2k,1,1}$.  Then $N(S)$ admits a Hamiltonian cycle.
\end{conjecture}

This article proves Conjecture~\ref{prelim:trichotomy}
(Theorem~\ref{assembly:theorem}).  The two excluded families do admit
a Hamiltonian path but no Hamiltonian cycle --- the binary family by
Theorem~\ref{inherited:binary-theorem}, the family $\sig{2k,1,1}$ by
the next proposition --- so the exception set is exact:

\begin{proposition}[$\sig{2k,1,1}$ is path-only]
\label{prelim:path-only-2k11}
For every $k \ge 1$, $N(2k,1,1)$ admits a Hamiltonian path but no
Hamiltonian cycle.
\end{proposition}

\begin{proof}
The signature has two odd multiplicities, so it has no stutters,
and $N(2k,1,1)$ equals $G(2k,1,1)$.  A Hamiltonian path is constructed by
Verhoeff \cite[\S5.4]{verhoeff2017}; the construction is also
recovered by the general path theory of
Stachowiak \cite{stachowiak1992}, whose main theorem covers every
signature with at least two odd multiplicities.

The absence of a Hamiltonian cycle is Stachowiak's
Lemma~12 \cite{stachowiak1992}; we give a short parity argument.  Let
$A$ be the set of words in which $1$ precedes $2$, and $B$ the set in
which $2$ precedes $1$.  The only edges between $A$ and $B$ are the
$2k+1$ swaps $0^i 1 2 0^{2k-i} \swap 0^i 2 1 0^{2k-i}$,
$0 \le i \le 2k$.  Their ends in $A$ have inversion count
$2(2k-i)$, so they all lie in the even class.  Replacing the $2$ by
a $1$ maps $A$ bijectively onto the words of $\sig{2k,2}$ and
preserves the inversion count, because within $A$ the $1$ and the
$2$ form no inversion.  By Lemma~\ref{prelim:stutter-parity}, $A$
therefore has $k+1$ more words in the even class than in the odd
class, one for each stutter of $\sig{2k,2}$ (the double $11$ can
occupy any of its $k+1$ dominoes).  Suppose a Hamiltonian
cycle uses $2t$ edges between $A$ and $B$.  It then covers $A$ by $t$
vertex-disjoint paths, each with both ends in $A$ on such an edge,
hence both ends even.  A path in a bipartite graph whose ends lie in
the same class has exactly one more vertex in that class than in the
other, so $t = k+1$.  The cycle would need $2k+2$ edges between $A$
and $B$, but there are only $2k+1$.  The impossibility halves of this proposition and of
Theorem~\ref{inherited:binary-theorem} are used only to show that the
exception set of Theorem~\ref{assembly:theorem} is minimal; the
constructions of this article consume only the existence halves.
\end{proof}

The remaining ingredient from the literature is Stachowiak's theorem
at two or more odd multiplicities, stated in
Section~\ref{sec:inherited} with Verhoeff's binary classification.  With two or more odd
multiplicities a signature has no stutters, so $N(S) = G(S)$ there.

\subsection{Parallel edges, gadgets, and gluing along a tree}
\label{prelim:parallel-section}

\begin{definition}[Parallel edges]\label{prelim:parallel-edges}
In $G(S)$, edges $a \swap b$ and $c \swap d$ are \emph{parallel} if
also $a \swap c$ and $b \swap d$ are edges.  The edges $a \swap c$ and
$b \swap d$ are then called the \emph{cross edges} of the parallel
pair, and the
four vertices $a,b,c,d$ with the four edges form a $4$-cycle.
\end{definition}

\begin{lemma}[Gluing a parallel pair; Verhoeff {\rm\cite[Lemma~11]{verhoeff2017}},
Stachowiak {\rm\cite{stachowiak1992}}]\label{prelim:gluing-lemma}
Let $C_1$ and $C_2$ be vertex-disjoint cycles in $G(S)$ carrying
parallel edges $a \swap b \in C_1$ and $c \swap d \in C_2$.  Then
replacing these two edges by the cross edges $a \swap c$ and
$b \swap d$ yields a single cycle on $V(C_1) \cup V(C_2)$.
\end{lemma}

\begin{proof}
$C_1$ minus $a \swap b$ is a Hamiltonian path in $C_1$ from $a$ to
$b$, and $C_2$ minus $c \swap d$ is one in $C_2$ from $c$ to $d$.
Concatenating $a \swap c$, the path from $c$ to $d$ in $C_2$,
$d \swap b$, and the path from $b$ to $a$ in $C_1$ gives a cycle on
all vertices of both.
\end{proof}

Iterating the gluing along a tree of parallel pairs yields the form in
which this article uses it.  Partition the vertex set of a graph into
\emph{units}, and let each unit $U$ carry a Hamiltonian cycle $H_U$ of
the subgraph it induces; an edge lying inside a unit is a \emph{unit
edge}.  A unit of two vertices, a \emph{digon}, carries its edge as a
doubled edge.  A \emph{gadget} is a $4$-cycle with two unit edges,
lying in two different units and on the cycles $H_U$ of those units,
and two \emph{cross edges}, each joining the two units; its unit edges
are pairs glued by the gluing lemma, its cross edges the edges added.
A family of gadgets forms a \emph{tree on the units} when, viewing
each gadget as an edge between its two units, they form a spanning
tree of the units.

\begin{lemma}[Gluing along a tree]\label{prelim:tree-gluing}
Let a graph have its vertex set partitioned into units, each carrying
a Hamiltonian cycle $H_U$.  Let $\mathcal{G}$ be a family of gadgets
forming a tree on the units, such that the unit edges of all gadgets
are pairwise distinct and a digon carries at most one of them.  Then
\[
  \bigcup_U H_U \;-\; \{\text{all unit edges}\} \;+\;
  \{\text{all cross edges}\}
\]
(as multisets; a digon with a unit edge contributes no remaining
edge) is a Hamiltonian cycle of the graph.
\end{lemma}

\begin{proof}
Add the gadgets in any order; in a tree, the two units of a gadget
are joined by no path of other gadgets, so they always lie in
different components.  Keep the invariant that every component of the
constructed subgraph is spanned by a single cycle, which contains
every unit edge of a later gadget inside it.  When the next gadget
with unit edges $e_1 = v_0 \swap v_1$ and $e_2 = v_2 \swap v_3$ joins
two components, delete $e_1$ from the cycle of the first and $e_2$
from the cycle of the second: each deletion turns a cycle into a
Hamiltonian path of that component, with ends $v_0,v_1$ respectively
$v_2,v_3$.  Adding the cross edges $v_1 \swap v_2$ and
$v_3 \swap v_0$ splices the two paths into one cycle.  Later unit
edges survive by distinctness; cross edges join different units, so
they lie on no $H_U$ and are never deleted.  A cross edge is added only once: a
gadget whose cross edge had been added earlier would have both its
units inside the component created by that earlier gadget, contrary
to the tree property.  At the end one cycle spans all units.
\end{proof}

The case of a single gadget is Lemma~\ref{prelim:gluing-lemma}.  The
lemma is used in two settings.  In the all-even case the units are the
hypercubes of Section~\ref{prelim:hypercubes}, each with a
boustrophedon cycle as its $H_U$ (Lemma~\ref{cube:boustrophedon}), and
the gadgets are the stitches of Section~\ref{prelim:hypercubes}.  In
the one-odd case and for the family $\sig{2j,2,1}$ the units are the
last-color blocks of Section~\ref{sec:one-odd}, and the gadgets are
the $4$-cycles of Lemma~\ref{one-odd:parallel}.  In the family one block is split
further, into a cycle and a digon (Proposition~\ref{family:prop});
that is the only digon in the article, and it carries exactly one
unit edge.

\subsection{Hypercubes: the orbits of the intra-domino swaps}
\label{prelim:hypercubes}

Throughout this subsection $S$ is a signature and $p$ a word of~$S$.
The \emph{content} of domino $i$ of $p$ is the ordered pair
$(p_{2i}, p_{2i+1})$ of the colors it carries.  The \emph{skeleton}
of $p$ records the content of each domino as an unordered pair of
colors; so $p$ is a stutter exactly when its skeleton consists of
doubles.  (The square and the $3$-cube of
Figure~\ref{intro:skeleton-figure} are examples of skeleton
classes.)

\begin{lemma}[Hypercubes]\label{cube:lemma1}
\mbox{}
\begin{enumerate}[label=\textup{(\alph*)},itemsep=0pt]
  \item The words with a given skeleton form a set of $2^r$
  words, where $r$ is the number of mixed dominoes.  The even
  edges make this set a copy of the hypercube $Q_r$ (the graph on the
  $2^r$ bit strings of length $r$, adjacent when they differ in one
  coordinate), with one coordinate per mixed domino, recording the order
  of its content.  We
  call the set a \emph{hypercube}, its coordinate directions the
  \emph{directions} of the hypercube, and $r$ its \emph{dimension}; the
  direction of a mixed domino $j$ is denoted $j$ as well.  Every odd
  edge joins two \emph{different} hypercubes.
  \item If all multiplicities of $S$ are even, no skeleton has exactly
  one mixed domino; hence $N(S)$ is the disjoint union of the hypercubes of
  dimension $r \ge 2$, together with the odd edges between them.
\end{enumerate}
\end{lemma}

\begin{proof}
(a) Words with the same skeleton differ only in the order of the two
colors within each mixed domino, and $s_{2i}$ flips domino $i$; the
content of each mixed domino can independently be ordered in two ways,
giving $2^r$ words, and the even swaps $s_{2i}$ at mixed dominoes are
the flips of the $r$ coordinates of a hypercube.  An odd swap at
position $2i+1$ exchanges two colors $b \ne c$ in
$\ldots a b \mid c d \ldots$ and replaces the domino contents
$\{a,b\},\{c,d\}$ by $\{a,c\},\{b,d\}$, and $\{a,c\} \ne
\{a,b\}$, so the skeleton changes and the swap leads to a different
hypercube.

(b) Every occurrence of a color $c$ lies either in a double $cc$ or in
a mixed domino containing $c$, so $k_c = 2 \cdot \#\{\text{doubles }cc\}
+ \#\{\text{mixed dominoes containing } c\}$.  If all $k_c$ are even,
every color lies in an even number of mixed dominoes, which is impossible
if there is exactly one mixed domino.  A non-stutter is a word
with at least one mixed domino; by what we just proved, in the all-even
case it has at least two, so it lies in a hypercube of dimension $r \ge 2$.
The hypercubes with $r = 0$ are the stutters, and those with $r = 1$ do not
occur; the even edges stay inside a hypercube and the odd edges join
different hypercubes, so the non-stutter graph is exactly the union of the
hypercubes with $r \ge 2$ plus the odd edges.
\end{proof}

\begin{definition}[Stitches]\label{cube:stitch}
Let shifted domino $i$ of $p$ be mixed (so $p_{2i+1} \ne p_{2i+2}$),
and let domino $j$ be mixed with $j \notin \{i, i+1\}$ (so domino $j$
is disjoint from the positions of the shifted domino).  Then the flips
$s_{2i+1}$ and $s_{2j}$ move disjoint positions and commute, so
\[
  q(p;i,j) \;=\; (\,p,\; s_{2j}(p),\;
    s_{2i+1}(s_{2j}(p)),\; s_{2i+1}(p)\,)
\]
is a $4$-cycle of the graph, a \emph{stitch}.  Its two edges in
direction $j$ are its \emph{hypercube edges}, lying in the hypercubes
of $p$ and of $s_{2i+1}(p)$; its two odd edges $s_{2i+1}$ are its
\emph{cross edges}.  Domino $j$ stays mixed in all four words, so all
four are non-stutters and a stitch is a gadget
(Lemma~\ref{prelim:tree-gluing}) with the two hypercubes as units and
its hypercube edges as unit edges.
\end{definition}

\begin{definition}[Rungs, anchored words]%
\label{cube:rung}
Let $C$ be a hypercube and let $l(C)$ be the index of its rightmost
mixed domino.
The edges of $C$ in direction $l(C)$ are its \emph{rungs}.  A word
$p$ is \emph{anchored} if every shifted domino is double, that is, if
it has no odd edge.  Stutters and anchored words are independent
notions: $0011$ is a stutter with an odd edge, and $0110$ is anchored
but not a stutter.
\end{definition}

\begin{lemma}[No stitch meets an anchored word]%
\label{cube:lemmaN}
No stitch contains an anchored word.  All four words of $q(p;i,j)$
carry two distinct colors at positions $2i+1, 2i+2$.
\end{lemma}

\begin{proof}
The words of $q(p;i,j)$ arise from $p$ by the swaps $s_{2j}$, which
does not touch positions $2i+1, 2i+2$, and $s_{2i+1}$, which exchanges
the two colors at those positions.  Since $p$ has distinct colors
there, so has every word of the stitch, and none of them is
anchored.
\end{proof}

\section{Inherited base results}
\label{sec:inherited}
This article relies on two earlier results, stated here with their
hypotheses and attribution.  The first supplies the two exception
families of the reformulated conjecture and is used only for their
exactness; the second is the article's only inherited Hamiltonicity
input, consumed as a black box.
\subsection{The binary classification (Verhoeff)}
\label{inherited:binary}

\begin{theorem}[Binary classification; Verhoeff
{\rm\cite{verhoeff2017}}]\label{inherited:binary-theorem}
For every binary signature $S = (k_0, k_1)$ with $k_0, k_1 \ge 1$,
the graph $N(S)$ admits a Hamiltonian path, and it admits a
Hamiltonian cycle exactly when $k_0$ and $k_1$ are both even.
\end{theorem}

The cycle clause is not an input of this article: an all-even binary
signature is covered by Theorem~\ref{thm:all-even}.  The path clause
and the no-cycle half of the odd--odd case are used only to show that
the exceptional families of the reformulated conjecture are exact
(Proposition~\ref{prelim:path-only-2k11} and
Remark~\ref{assembly:exact}); the cycle constructions of this article
consume neither.

The existence clauses are machine-checked twice over: a deterministic
implementation of Verhoeff's recursion in the Python package
distributed with this submission, checked by its standing gate,
builds the path or cycle at every binary signature of length at most
$14$, and the construction of the path and of the even--even cycle is
formalized in the Lean development
(Appendix~\ref{appendix:formalization}).

\subsection{The multi-odd case (Stachowiak)}
\label{inherited:stachowiak-section}

\begin{theorem}[Multi-odd signatures; Stachowiak
{\rm\cite{stachowiak1992}}, as recorded by Verhoeff
{\rm\cite{verhoeff2017}}]\label{inherited:stachowiak}
Let $S$ be a signature of arity at least $3$ with at least two odd
multiplicities that is not a permutation of $\sig{2j,1,1}$ for any
$j$.  Then $G(S)$ admits a Hamiltonian cycle.
\end{theorem}

A stutter (Section~\ref{prelim:stutter}) consists of doubles and at
most one unpaired trailing color, so its signature has at most one
odd multiplicity.  Hence a signature with at least two odd
multiplicities has no stutter, $N(S) = G(S)$, and the theorem says
equivalently that $N(S)$ admits a Hamiltonian cycle.

\begin{remark}[Attribution]\label{inherited:stachowiak-attribution}
The cycle is Stachowiak's construction for Hamiltonian paths in
graphs of linear extensions of unions of posets
\cite{stachowiak1992}; the statement above is the form in which
Verhoeff \cite{verhoeff2017} records it.  Stachowiak's main theorem
characterizes the signatures with a Hamiltonian path (either the
signature is linear or at least two multiplicities are odd); the
cycle form is his Lemmas~6 and~11, whose hypotheses we have checked
against the printed paper.  For the exactness of the exception set
(Proposition~\ref{prelim:path-only-2k11}) we also use his main
theorem in its path form and his Lemma~12, which excludes a cycle
at $\sig{2k,1,1}$.
\end{remark}

\begin{remark}[Consumption in this article]%
\label{inherited:stachowiak-consumption}
Theorem~\ref{inherited:stachowiak} is the article's single inherited
Hamiltonicity input.  It is consumed at three places, the first two
of them graphs with exactly two odd multiplicities:
\begin{enumerate}[label=\textup{(\arabic*)},itemsep=1pt]
  \item the blocks $T_z = G(S^z)\cdot z$ of the one-odd section
  (Proposition~\ref{one-odd:assembly}(a)), whose cycle is used as an
  arbitrary Hamiltonian cycle --- nothing beyond its existence is
  assumed;
  \item the block $T_0 = G(2j-1,2,1)\cdot 0$ of the family
  $\sig{2j,2,1}$ (Proposition~\ref{family:prop});
  \item the main theorem at two or more odd multiplicities
  (Theorem~\ref{assembly:theorem}), where $N(S) = G(S)$ and the
  conclusion is the theorem itself.
\end{enumerate}
The two-odd graphs with arity $2$ are excluded by the arity hypothesis
(and are the odd--odd binary exception,
Theorem~\ref{inherited:binary-theorem}); the family $\sig{2j,1,1}$ is
excluded in the statement and avoided in the applications by the
argument in Proposition~\ref{one-odd:assembly}(a).  The theorem is
stated only here, and these three places are its only uses.  Both supporting artifacts of the article cover this input:
the Python implementation with its standing gate, and the Lean
formalization with the bridges (Appendices~\ref{appendix:machine-verification}
and~\ref{appendix:formalization}).
\end{remark}

\section{Two tools: hypercube cycles and spanning matching trees}
\label{sec:tools}

The all-even proof glues the hypercubes of Lemma~\ref{cube:lemma1} along a
spanning tree of stitches.  The two inputs to that gluing are the
following lemma, which supplies a Hamiltonian cycle of each hypercube, and
Lemma~\ref{tools:lemmaM}, which supplies a set of pairwise vertex-disjoint
odd edges of $G(S)$ --- a matching --- whose image between hypercubes is a
spanning tree of the hypercubes, and which can be chosen so that every
color has a free word (free is defined by the invariant in the lemma).

\begin{lemma}[Boustrophedon hypercube cycles]\label{cube:boustrophedon}
Let $r \ge 2$, let $d$ be a direction of the hypercube $Q_r$, and let
$f$ be an edge of $Q_r$ not in direction $d$.  Then $Q_r$ has a
Hamiltonian cycle that contains all $2^{r-1}$ edges in direction $d$,
and $f$.  We call it a \emph{boustrophedon cycle in direction $d$}.
\end{lemma}

\begin{proof}
Induction on $r$.  $Q_2$ is a $4$-cycle, which contains every edge.
For $r \ge 3$ write $Q_r = Q_{r-1} \times \{0,1\}$, with the last
factor in direction $d$, and let $f = \{(u,0),(u',0)\}$ (exchange the
two levels if necessary).  By the lemma for $Q_{r-1}$, in the
direction of the edge $uu'$ and with any edge in another direction,
$Q_{r-1}$ has a Hamiltonian cycle that contains $uu'$; write it as
$u_0 u_1 \cdots u_{2^{r-1}-1}$ with $u_0 = u$ and $u_1 = u'$.  The
walk
\[
  (u_0,1),\ (u_0,0),\ (u_1,0),\ (u_1,1),\ (u_2,1),\ (u_2,0),\ \ldots
\]
uses the edge in direction $d$ at every $u_t$, and it uses the edge
$u_t u_{t+1}$ at level $0$ for even $t$ and at level $1$ for odd $t$.
Since $2^{r-1}$ is even, the walk closes up; it is Hamiltonian and
contains $f$ (the step $t=0$).
\end{proof}

The binary reflected Gray code, read as a cycle, is a boustrophedon
cycle in the direction of its fastest-changing coordinate, which flips
at every other step.  The lemma adds the freedom to pass through one
prescribed edge in another direction, and the proof needs no Gray
code: the cycle of $Q_{r-1}$ comes from the lemma itself.

\begin{lemma}[Spanning matching trees]\label{tools:lemmaM}
Let $S$ be a signature of even length $2m \ge 2$.  There exists a set
$M(S)$ of odd edges of $G(S)$ that is a \emph{matching} (the edges are
pairwise vertex-disjoint) and, viewed as edges between the hypercubes of
$V(S)$ (the $0$-dimensional ones, the stutters, included), a
\emph{spanning tree} of those hypercubes.  The induction of the proof
establishes a strengthening of this claim, as is common: the set can be
chosen so that the following invariant holds.
\begin{itemize}[itemsep=0pt]
  \item[\textup{(INV)}] for every color $z$ of $S$, some word
  ending in $z$ lies on no edge of $M(S)$.  We call such a
  word \emph{free}.
\end{itemize}
\end{lemma}

\begin{proof}
Induction on $m$.  If $m = 1$, or $S$ has one color, there is a
single hypercube and $M(S) = \emptyset$.  Let $m \ge 2$, let
$a < b$ be the two smallest colors of $S$, and let $R$ be
the set of the other colors.

\emph{Blocks and groups.}  For colors $c \ne d$ put
\[
  X_{cd} = \{\, t \cdot cd,\ t \cdot dc \;:\; t \in V(S - c - d) \,\},
\]
and if $k_c \ge 2$ put
$X_{cc} = \{\, t \cdot cc : t \in V(S - c - c) \,\}$.  These are the
\emph{blocks}; the block $X_{cd}$ depends only on the pair $\{c,d\}$,
so $X_{cd} = X_{dc}$.
The content of the last domino is part of the skeleton, so each hypercube
lies in exactly one block, and $t \mapsto t \cdot cd$ is a bijection
between the hypercubes of $V(S-c-d)$ and the hypercubes of $X_{cd}$.  Odd
edges of $G(S-c-d)$ never move the last color of $t$; that color
is the \emph{group} of $t \cdot cd$.

\emph{Lifts.}  For every non-double block $X_{cd}$ and every group
$g$, choose an orientation $o(cd,g) \in \{cd, dc\}$ (fixed below).
Lift every edge $\{t_1, t_2\}$ of $M(S-c-d)$ whose two
words have group $g$ to $\{t_1 \cdot o, t_2 \cdot o\}$, where
$o = o(cd,g)$, and lift every edge of $M(S-c-c)$ to
$\{t_1 \cdot cc,\, t_2 \cdot cc\}$.  Inside each block the lifts form
a matching and a spanning tree of the block's hypercubes, by the induction
hypothesis.

\emph{Bridges.}  For $c \ne d$, the odd swap at position $2m-3$ maps
$v \cdot c \cdot dc'$ to $v \cdot d \cdot cc'$, and so joins the block
$X_{dc'}$ to the block $X_{cc'}$.  Add the following bridges; each is
\emph{owned} by the block named first:
\begin{itemize}[itemsep=1pt]
  \item[\textup{(T1)}, \textup{(T2)}] for $d \in R$: let $t_d = v_d \cdot d$ be a free word
  of $V(S-a-b)$ (INV).  Add $v_d d \cdot ba \swap
  v_d b \cdot da$, joining $X_{ad}$ to
  $X_{ab}$, and $v_d d \cdot ab \swap v_d a
  \cdot db$, joining $X_{bd}$ to $X_{ab}$;
  \item[\textup{(T3)}] for $c < d$ in $R$: add $v a \cdot cd \swap v c \cdot
  ad$, joining $X_{cd}$ to $X_{ad}$, for any
  $v \in V(S-a-c-d)$;
  \item[\textup{(T4)}] for $k_c \ge 2$: put $y_c = a$ if $c \ne a$, and
  $y_a = b$.  Let $v \cdot y_c$ be a free word of
  $V(S-c-c)$ (INV), and add $v y_c \cdot cc \swap v c \cdot y_c c$,
  joining $X_{cc}$ to $X_{y_c c}$.
\end{itemize}
Every block other than $X_{ab}$ owns exactly one bridge.  Its
bridge leads to $X_{ab}$, or to a block $X_{ad}$ whose
own bridge leads to $X_{ab}$, so the bridges form a spanning
tree of the blocks.  Lifts stay inside blocks, and the bridges join
pairwise different pairs of blocks, so no two edges of $M(S)$ join
the same two hypercubes, and the lifts together with the bridges form a
spanning tree of the hypercubes (Figure~\ref{tools:block-tree}).

\begin{figure}[ht]
\centering
\begin{tikzpicture}[
  blk/.style={draw, rounded corners=2pt, inner sep=3pt},
  lbl/.style={font=\scriptsize, fill=white, inner sep=1pt}]
  \node[blk] (ab) at (0,0) {$X_{ab}$};
  \node[blk] (aa) at (-4.2,-1.5) {$X_{aa}$};
  \node[blk] (bb) at (-2.1,-1.5) {$X_{bb}$};
  \node[blk] (bd) at (0.6,-1.5) {$X_{bd}$};
  \node[blk] (ad) at (3.6,-1.5) {$X_{ad}$};
  \node[blk] (cd) at (2.4,-3.0) {$X_{cd}$};
  \node[blk] (dd) at (4.8,-3.0) {$X_{dd}$};
  \draw (aa) -- node[lbl] {T4} (ab);
  \draw (bb) -- node[lbl] {T4} (ab);
  \draw (bd) -- node[lbl] {T2} (ab);
  \draw (ad) -- node[lbl] {T1} (ab);
  \draw (cd) -- node[lbl] {T3} (ad);
  \draw (dd) -- node[lbl] {T4} (ad);
\end{tikzpicture}
\caption{The tree of blocks in the proof of Lemma~\ref{tools:lemmaM},
for colors $c < d$ in $R$.  Each block owns the bridge to its parent;
$X_{aa}$, $X_{bb}$ and $X_{dd}$ occur only when the color is
repeated.}
\label{tools:block-tree}
\end{figure}
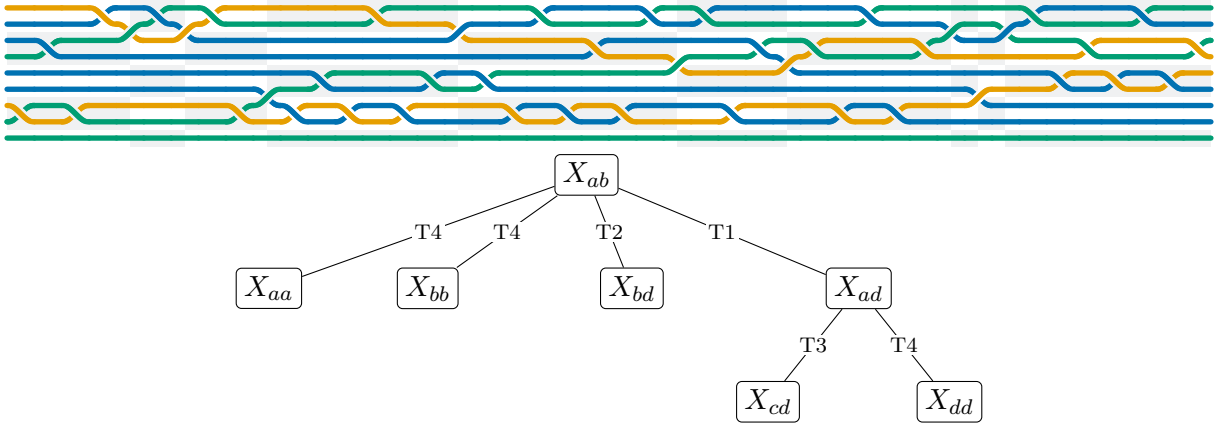

\emph{Matching.}  Record each bridge word in a non-double block
as (block; group, last domino):

\begin{center}
\footnotesize
\begin{tabular}{l|l}
\hline
block & records of bridge words\\
\hline
$X_{ab}$ & $(d; ba)$, $(d; ab)$ for
$d \in R$ (prefix $t_d$, free); $(a; ba)$ if
$k_a \ge 2$; $(b; ab)$ if $k_b \ge 2$\\
$X_{ad}$, $d \in R$ & $(b; da)$; $(c; ad)$ for
$c \in R$, $c < d$; $(d; ad)$ if $k_d \ge 2$\\
$X_{bd}$, $d \in R$ & $(a; db)$\\
$X_{cd}$, $c < d$ in $R$ & $(a; cd)$\\
\hline
\end{tabular}
\end{center}

Apart from the groups $d \in R$ of $X_{ab}$, each
(block, group) holds at most one record.  In $X_{ab}$, group
$a$ holds only $(a; ba)$ and group $b$ only
$(b; ab)$.  In $X_{ad}$ the groups $b$,
$c < d$ and $d$ of the records are distinct, and $X_{bd}$ and
$X_{cd}$ hold one record each.

\emph{Principle.}  For a (block, group) that holds one record, choose
$o(cd,g)$ to be the orientation that does \emph{not} occur in that
record.  For every other (block, group), the groups $d$ of
$X_{ab}$ included, choose $o(cd,g)$ arbitrarily.  Then no
bridge word lies on a lift: the lifts of a group that holds
one record have the orientation that the record does not have, so
they avoid the recorded word; the bridge words of group
$d \in R$ in $X_{ab}$ are $t_d \cdot ba$ and
$t_d \cdot ab$, and $t_d$ lies on no edge of
$M(S-a-b)$, so no lift contains them; and a double block
$X_{cc}$ holds exactly one bridge word, its own (T4)
word $v y_c \cdot cc$, whose prefix is free by choice.  Bridge words are pairwise
distinct: the records in the table differ, and each double block
holds only one.  So $M(S)$ is a matching.

\emph{The invariant for $S$.}  In each row of the following table,
$t$ is free in the smaller matching (by INV) and the record of
$t \cdot p$ is not in the table above, so $t \cdot p$ is free in
$M(S)$.

\begin{center}
\footnotesize
\begin{tabular}{l|l|l}
\hline
color $z$ & free word $t \cdot p$ & why it is free\\
\hline
$k_z \ge 3$ & $t \cdot zz$, $t$ ends in $z$ & the bridge of
$X_{zz}$ has prefix ending in $y_z \ne z$\\
$z \in R$ & $t \cdot az$, $t$ ends in $b$ & group $b$
of $X_{az}$ holds only $(b; za)$\\
$z = b$, $k_a \ge 2$ & $t \cdot ab$, $t$ ends in
$a$ & $(a; ab)$ is not a record\\
$z = b$, $k_a = 1$, & $t \cdot db$, $d \in R$, &
$X_{bd}$ holds only\\
$k_b \le 2$ & $t$ ends in $g \ne a$ & $(a; db)$\\
$z = a$, $k_b \ge 2$ & $t \cdot ba$, $t$ ends in
$b$ & $(b; ba)$ is not a record\\
$z = a$, $k_b = 1$, & $t \cdot da$, $d \in R$, &
group $g$ of $X_{ad}$ has no\\
$k_a \le 2$ & $t$ ends in $g \ne b$ & record
$(g; da)$\\
\hline
\end{tabular}
\end{center}

In the last two rows, $R$ is non-empty because $\lvert S \rvert \ge
4$, and a color $g$ with the required property exists because
$\lvert S - b - d \rvert \ge 2$ (respectively
$\lvert S - a - d \rvert \ge 2$) while $a$ (respectively
$b$) occurs only once.  The rows are exhaustive.
\end{proof}

\section{The all-even case: all multiplicities even}
\label{sec:all-even}

\begin{theorem}[All-even theorem]\label{thm:all-even}
Let $S$ be a signature of arity at least $2$ in which every multiplicity
is even.  Then $N(S)$ admits a Hamiltonian cycle.
\end{theorem}

Throughout this section $S$ is all even and $n = 2m$; its colors are
named so that $0 < 1$ are the two smallest.  By
Lemma~\ref{cube:lemma1}(b) the non-stutter graph is the union of the
hypercubes of dimension $r \ge 2$ together with the odd edges.  We build
the cycle by gluing the hypercubes along a spanning tree of stitches.  The
stitches come in three kinds: the \emph{rung stitches} form one
spanning tree inside every \emph{family} of hypercubes, and each family
other than the root family owns one connector, a \emph{descent} or a
\emph{root stitch}, that attaches it to a family closer to the root.

\subsection{Families and rung stitches}

Let $p$ be a non-stutter of $S$, let $l$ be its rightmost mixed
domino, let $\{a,b\}$ with $a < b$ be its content, and let
$u = p_{2l+2} \cdots p_{n-1}$ be its trailing word (a concatenation
of doubles).  The \emph{family} of $p$ is
\[
  F(u, ab) = \{\, t \cdot p' \cdot u \;:\;
  t \in V(S - u - a - b),\ p' \in \{ab, ba\} \,\}.
\]
When the order of the two colors is not known we write
$F(u, \{a,b\})$.  In $S - u - a - b$ the colors $a$ and $b$ have odd
multiplicities, so $l \ge 1$.  The families partition the hypercubes: a
hypercube is determined by its rightmost mixed domino, its trailing word, and
the hypercube of the prefix.  The hypercube of $t \cdot p' \cdot u$ is
the hypercube of $t$ times $\{ab, ba\}$, and its rungs are the edges
$\{t \cdot ab \cdot u,\; t \cdot ba \cdot u\}$.

For each edge $\{t_1, t_2 = s_{2i+1}(t_1)\}$ of the matching
$M(S - u - a - b)$ of Lemma~\ref{tools:lemmaM}, take the stitch
$q(t_1 \cdot ab \cdot u;\, i, l)$.  It is a stitch since
$i+1 \le l-1$ and $l$ is a mixed domino of $t_1 \cdot ab \cdot
u$.  By Lemma~\ref{tools:lemmaM} these \emph{rung stitches}
of a family form a spanning tree of the family's hypercubes, and they meet
each hypercube in pairwise distinct rungs.  The smallest
instance is Figure~\ref{intro:stitch-figure}.

\subsection{Descent and root stitches}

Two further kinds of stitches connect the families.  In both of them
the exchanged colors differ and the direction is a mixed domino lying
outside the two exchanged dominoes, so it is a stitch of
$N(S)$;
and in each of its two hypercubes it uses an edge in a direction that is
never a rung, since the rung directions of a family are $l$ and
$l+1$, respectively $m-1 \ge 2$ in the root case below.  We describe
them by the family they start in; the stitch leads to a family closer
to the root family $L_{01} = F(\emptyset, 01)$.

\emph{Descent stitches} ($u = (yy)u'$ non-empty).  Let $e$ be
the \emph{largest} color of $\{a,b\} \setminus \{y\}$ and let $e'$ be
the other color of $\{a,b\}$.  (Taking the largest color matters
only in Lemma~\ref{all-even:lemmaI}, where it keeps the descent
stitches apart from the root stitches.)  Let $w$ be any stutter of
$S - u - ab - ab$ (for instance the sorted one) and put
\[
  v_F = w \cdot ab \cdot e'e \cdot yy \cdot u',
  \qquad
  v'_F = w \cdot ab \cdot e'y \cdot ey \cdot u'.
\]
The stitch $q(v_F;\, l, l-1)$ swaps $e$ with $y$ and lands in $v'_F$,
which lies in the family $F(u', \{e,y\})$, a family with one
double fewer; Figure~\ref{intro:stitch-figure} shows this
stitch for $u = (00)$.

\emph{Root stitches} ($u$ empty; $L_{ab} = F(\emptyset, ab)$ with
$\{a,b\} \ne \{0,1\}$).  Let $c$ be the smallest color not in
$\{a,b\}$; thus $c = 0$ if $a \ne 0$, and $c = 1$ if $a = 0$ (then
$b \ge 2$ because $\{a,b\} \ne \{0,1\}$).  Let $w'$ be any stutter
of $S - aa - bb - cc$ and put
\[
  v_{ab} = ab \cdot w' \cdot cc \cdot ba,
  \qquad
  v'_{ab} = ab \cdot w' \cdot cb \cdot ca.
\]
The stitch $q(v_{ab};\, m-2, 0)$ swaps $c$ with $b$ and lands in
$v'_{ab}$, which lies in the root family with last domino $\{c,a\}$:
$L_{0a}$ if $a \ne 0$, and $L_{01}$ if $a = 0$.  Arity at least $3$
gives the color $c$ and $m \ge 3$; since every multiplicity is even
and $c \notin \{a,b\}$, the stutter $w'$ exists.

Every family other than $L_{01}$ owns exactly one such stitch: a
descent stitch if $u$ is non-empty, and a root stitch otherwise.
It leads to a family with shorter $u$, or from a root family
$L_{ab}$ to $L_{0a}$ or $L_{01}$, so every root family reaches $L_{01}$
in at most two steps.  Together with the rung stitches, the descent
and root stitches therefore form a spanning tree of all hypercubes.  In
arity $2$ all families are $F(u, 01)$, the set
$\{a,b\} \setminus \{y\}$ has one color, and there are no root
stitches; the tree then consists of the rung stitches alone.

\begin{lemma}\label{all-even:lemmaI}
Every hypercube meets at most one descent or root stitch.
\end{lemma}

\begin{proof}
The hypercubes involved are those of the words $v_F$, $v'_F$,
$v_{ab}$, $v'_{ab}$, so it suffices to show that these words
have pairwise distinct skeletons.  In such a word let $\Lambda$
be the rightmost mixed domino and $\Lambda'$ the domino before it,
both taken as unordered pairs.  Both, and the doubles after $\Lambda$, can be
read off the skeleton.  By construction:
\begin{itemize}[itemsep=0pt]
  \item in $v_F$, $\Lambda' = \Lambda = \{a,b\}$;
  \item in $v_{ab}$, $\Lambda' = cc$ with $c \notin \Lambda = \{a,b\}$;
  \item in $v'_F$, $\Lambda' = \{e',y\}$ and $\Lambda = \{e,y\}$ have
  exactly the color $y$ in common;
  \item in $v'_{ab}$, $\Lambda' = \{c,b\}$ and $\Lambda = \{c,a\}$
  have exactly the color $c$ in common.
\end{itemize}
So a $v_F$, a $v_{ab}$, and a $v'_F$ or $v'_{ab}$ never share a
skeleton.  The
$v_F$ lie in distinct families $F$, so they are pairwise distinct, and
so are the $v_{ab}$.  For a word of the last two kinds, let $y$
be the color common to $\Lambda'$ and $\Lambda$, let $e$ be the other
color of $\Lambda$, let $e'$ be the other color of $\Lambda'$
($e' = y$ if $\Lambda'$ is a double), and let $u'$ be the trailing
word after $\Lambda$.  If the word is $v'_F$, then
$F = F((yy)u', \{e,e'\})$.  If it is $v'_{ab}$, then $u'$ is
empty and $(c,a,b) = (y,e,e')$.  So the $v'_F$ are pairwise distinct,
and so are the $v'_{ab}$.  Finally, $v'_F = v'_{ab}$ would give
$F = F(cc, ab)$ with $c \notin \{a,b\}$ and $e = a < b$.  But $e$ is
the largest color of $\{a,b\} \setminus \{c\} = \{a,b\}$, which is
$b$.
\end{proof}

\begin{proof}[Proof of Theorem~\ref{thm:all-even}]
Let $\mathcal{G}$ be the set of all rung, descent and root stitches.
As a graph on the hypercubes it is a spanning tree, which is the tree
condition of Lemma~\ref{prelim:tree-gluing}.  The rung stitches meet
each hypercube in pairwise distinct rungs (Lemma~\ref{tools:lemmaM}), and
at most one further stitch meets it, in a non-rung edge
(Lemma~\ref{all-even:lemmaI}).  For a hypercube met by such a non-rung
stitch edge $f$, choose its boustrophedon cycle in the rung direction
through $f$ (Lemma~\ref{cube:boustrophedon}); for every other hypercube
choose any boustrophedon cycle in its rung direction.  Each such
cycle contains every rung of its hypercube and the at most one non-rung
unit edge.  By the two facts just cited, and because a non-rung edge
is never a rung, the unit edges of the stitches are pairwise distinct.
The gluing lemma therefore applies, and Lemmas~\ref{cube:boustrophedon}
and~\ref{prelim:tree-gluing} give a Hamiltonian cycle of $N(S)$.
\end{proof}

Every choice above can be fixed by an explicit rule --- for instance
the lexicographically first free word in
Lemma~\ref{tools:lemmaM}, and the sorted stutters $w$, $w'$ ---
and any admissible choice works.  The one-odd section uses one more
fact about the stitches.

\begin{proposition}[Property \textup{(P)}]\label{all-even:propertyP}
Every hypercube met by a descent or root stitch has at most three mixed
dominoes.  Hence a hypercube of dimension at least $4$ is met only by rung
stitches, and only in rungs.
\end{proposition}

\begin{proof}
$v_F$ has the mixed dominoes $ab$ and $e'e$; $v'_F$ has $ab$, $ey$, and
possibly $e'y$; $v_{ab}$ has two ($ab$ and $ba$); and $v'_{ab}$ has
three ($ab$, $cb$, $ca$, with $c \notin \{a,b\}$).  The hypercube of such a
word has exactly these mixed dominoes, since $w$ and $w'$ are
stutters and the remaining dominoes are doubles.
\end{proof}

\section{The one-odd case: exactly one odd multiplicity}
\label{sec:one-odd}

\begin{theorem}[One-odd theorem]\label{thm:one-odd}
Let $S$ be a signature of arity at least $3$ with exactly one odd
multiplicity.  Then $N(S)$ admits a Hamiltonian cycle.
\end{theorem}

Throughout this section $S$ has arity at least $3$, exactly one odd
multiplicity, at the color $x$, and $n = \lvert S \rvert$ is odd.
The signature $S^x = S - x$ (Section~\ref{prelim:signatures}) is all
even of length $n-1$.  Let $Z$ be the set of colors of $S$ other than
$x$.  Every $z \in Z$ has $k_z \ge 2$.  The arity of
$S^x$ equals $K(S)$ if $k_x \ge 3$, and $K(S) - 1$ if $k_x = 1$;
so $K(S^x) \ge 2$.  Write $n = 2m+1$, so that $S^x$ has length
$2m$, and let $\rho = s_{n-2}$, the swap of the last two colors.
The proof splits $N(S)$ by the last color, uses the all-even theorem
inside the block ending in $x$, forces one edge inside each block
ending in another color, and glues with parallel pairs whose
prescribed edge is read off the all-even cycle.

\subsection{Blocks and parallel pairs}

For a color $c$, the \emph{block} of $c$ consists of the non-stutter
words of $S$ ending in $c$; deleting that final $c$ maps the
block bijectively to its \emph{shortened form}, a set of words
of $S^c$.  The next lemma identifies the shortened forms and shows
that the blocks are joined only by the swaps of the last two colors.

\begin{lemma}[Last-color blocks]\label{one-odd:blocks}
The non-stutters of $S$ ending in $x$ form the block
$T_x = N(S^x) \cdot x$.  For $z \in Z$, the non-stutters ending in
$z$ form the block $T_z = G(S^z) \cdot z$, and $S^z$ has exactly two
odd multiplicities, $k_x$ and $k_z - 1$.  The edges of $N(S)$ between
different blocks are exactly the $\rho$-edges.
\end{lemma}

\begin{proof}
Since $n$ is odd, the last color is unpaired and $u \cdot c$ is a
stutter exactly when $u$ is.  The signature $S^x$ is all even, so
$T_x = N(S^x) \cdot x$.  For $z \in Z$ the signature $S^z$ has
exactly two odd multiplicities, so it has no stutters and
$T_z = G(S^z) \cdot z$.  A swap $s_p$ with $p \le n-3$ keeps the
last color and so stays inside a block; $\rho = s_{n-2}$ changes the
last color, so every edge between different blocks is a $\rho$-edge.
Conversely, a $\rho$-edge between two blocks is an edge of $N(S)$:
both its ends are non-stutters, because a word ending in $x$
is a stutter exactly when its prefix in $S^x$ is, and a
word ending in some $z \in Z$ lies in $S^z$, which has no
stutters.
\end{proof}

\begin{lemma}[Parallel pairs]\label{one-odd:parallel}
Let $w$ be a word of $S$ with $w_{n-2} \ne w_{n-1}$, and let
$i \le n-4$ with $w_i \ne w_{i+1}$.  Then $s_i$ and $\rho$ commute,
and $w$, $s_i(w)$, $\rho(s_i(w))$, $\rho(w)$ is a $4$-cycle of
$G(S)$.  Its edges $f = \{w, s_i(w)\}$ and
$\rho(f) = \{\rho(w), \rho(s_i(w))\}$ form a parallel pair, with
cross edges $\{s_i(w), \rho(s_i(w))\}$ and $\{\rho(w), w\}$.  If all four words are non-stutters and $f$
and $\rho(f)$ lie in different units, this $4$-cycle is a gadget in
the sense of Lemma~\ref{prelim:tree-gluing}, with unit edges $f$ and
$\rho(f)$.
\end{lemma}

\begin{proof}
The swaps $s_i$ and $\rho$ move disjoint positions, and each of them
exchanges two distinct colors in every word involved.  Each
cross edge joins an end of $f$ to an end of $\rho(f)$.
\end{proof}

\subsection{Shifted stutters and the assembly}

For $z \in Z$ let
$w = e_1e_1 e_2e_2 \cdots e_{m-1}e_{m-1}$ be a stutter of
$S^x - zz$ whose last color $\omega = e_{m-1}$ is not $z$.  Such a
stutter exists because $S^x - zz$ contains a color other than $z$.
Put
\[
  v'_z = z \cdot w \cdot z, \qquad v_z = z \cdot w \cdot x \cdot z,
\]
so that $v'_z$ is a word of $S^x$, $v_z$ is one of
$T_z$, and $\rho(v_z) = v'_z \cdot x$.  We call $v'_z$ a \emph{shifted
stutter of type $z$}, and write $C_z$ for its hypercube in $N(S^x)$.

\begin{lemma}[Shifted stutters]\label{one-odd:shifted}
\mbox{}
\begin{enumerate}[label=\textup{(\roman*)},itemsep=2pt]
  \item $v'_z$ is anchored, and domino $i$ carries $(e_i, e_{i+1})$ for
  $0 \le i \le m-1$, where $e_0 = e_m = z$.  Hence, by the
  \emph{dimension rule}, the dimension of $C_z$ is the number of
  color changes of the cyclic sequence
  $(z, e_1, \ldots, e_{m-1})$, and
  $\dim C_z \ge K(S^x) \ge 2$.  The last domino $(\omega, z)$ is
  mixed, so the rungs of $C_z$ are its edges in direction $m-1$.
  \item The edges of $v_z$ in $T_z$ are the swap $s_{n-3}$ of
  $\omega$ and $x$, present if and only if $\omega \ne x$ (the
  \emph{exit}), and, for each non-last direction $i$ of $C_z$
  ($i < m-1$, $e_i \ne e_{i+1}$), the swap $s_{2i}$.  For such an
  edge $f = \{v_z, s_{2i}(v_z)\}$ we have $\rho(f) = g_i \cdot x$,
  where $g_i = \{v'_z, s_{2i}(v'_z)\}$ is the edge of $C_z$ at $v'_z$
  in direction $i$; the four words are non-stutters, so
  $(f, g_i \cdot x)$ is a parallel pair meeting the gadget condition
  of Lemma~\ref{one-odd:parallel} whenever $f$ and $g_i \cdot x$ lie
  in different units.  We call $g_i$ the \emph{partner} of $f$.
  \item $\deg_{T_z}(v_z) = (\dim C_z - 1) + [\omega \ne x]$.  Every
  Hamiltonian cycle of $T_z$ uses a non-exit edge at $v_z$, and uses
  both edges at $v_z$ if the degree is $2$.
\end{enumerate}
\end{lemma}

\begin{proof}
In $v'_z = z e_1e_1 \cdots e_{m-1}e_{m-1} z$ the positions $2i+1$ and
$2i+2$ carry $e_{i+1}e_{i+1}$, so $v'_z$ is anchored, and domino $i$
carries $(e_i, e_{i+1})$.  The sequence $z, e_1, \ldots, e_{m-1}$ uses every
color of $S^x$, so the cyclic sequence has at least $K(S^x)$
changes and $\dim C_z \ge K(S^x) \ge 2$.

In $v_z = z e_1e_1 \cdots e_{m-1}e_{m-1} x z$ the swaps $s_p$ with
$p \le n-3$ that exchange distinct colors have even $p$.  For
$i < m-1$ the swap $s_{2i}$ exchanges $e_i$ and $e_{i+1}$, as in
$v'_z$; note $2i \le 2m-4 = n-5$ here.  The swap $s_{2m-2} = s_{n-3}$
exchanges $\omega = e_{m-1}$ with $x$, which stands at position
$2m-1 = n-2$, where $v'_z$ has $e_m = z$.  For $2i \le n-5$,
Lemma~\ref{one-odd:parallel} applies, because $v_z[n-2] = x \ne z$.
The words of $f$ lie in $T_z$, which has no stutters, and
those of $g_i$ lie in the hypercube $C_z$ of the all-even signature
$S^x$.  For (iii): a Hamiltonian cycle has two edges at $v_z$, at
most one of which is the exit.
\end{proof}

\begin{proposition}[Assembly]\label{one-odd:assembly}
Let $S$ be as in Theorem~\ref{thm:one-odd}, not a permutation of $\sig{2j,2,1}$.
\begin{enumerate}[label=\textup{(\alph*)},itemsep=1pt]
  \item Each block $T_z$, $z \in Z$, has a Hamiltonian cycle, by the
  inherited theorem (Theorem~\ref{inherited:stachowiak}).
  \item For each $z \in Z$ let $K_z$ be any Hamiltonian cycle of
  $T_z$, and let $v'_z$ be a shifted stutter of type $z$.  Suppose
  $N(S^x)$ has a Hamiltonian cycle $\Theta$ that contains, for
  every $z$, the partner $g_z$ of some non-exit edge $f_z$ of $K_z$
  at $v_z$.  Then $N(S)$ has a Hamiltonian cycle.
\end{enumerate}
\end{proposition}

\begin{proof}
\emph{(a)} For $z \in Z$ the signature $S^z$ has exactly two odd
multiplicities, $k_x$ and $k_z - 1$.  Since $k_z \ge 2$, it has the
support of $S$, so its arity is at least $3$.  If it were a
permutation of $\sig{2i,1,1}$, then $k_x = 1$ and $k_z = 2$ with
arity $3$, which would make $S$ a permutation of $\sig{2i,2,1}$.
So the inherited theorem applies.

\emph{(b)} The units are the block $T_x$ carrying the cycle
$\Theta \cdot x$ and the blocks $T_z$ carrying the cycles $K_z$
(Lemma~\ref{one-odd:blocks}).  The gadgets are the $4$-cycles of the
parallel pairs $(f_z, g_z \cdot x)$ of
Lemma~\ref{one-odd:parallel} and Lemma~\ref{one-odd:shifted}(ii).
They form a star centered at $\Theta \cdot x$: every gadget has its
edge $f_z$ on the unit $K_z$ and its other unit edge $g_z\cdot x$ on
$T_x$.  The edges $g_z \cdot x$ are pairwise distinct because their
last two colors are $z\,x$ for distinct $z$, and each $K_z$ carries
exactly one unit edge.  Lemma~\ref{prelim:tree-gluing} gives one
Hamiltonian cycle.
\end{proof}

\subsection{Choosing the shifted stutters}

We choose $v'_z$ for all $z \in Z$ as follows; the rule is stated for
$S$ up to renaming colors, and the number of changes (the dimension
rule) can be read off each word.

\begin{itemize}[itemsep=2pt]
  \item \emph{$K(S^x) \ge 3$ and $S^x \ne \sig{2,2,2}$:} choose
  any shifted stutter whose cyclic sequence has at least $4$ changes,
  so that $\dim C_z \ge 4$.  One exists:
  \begin{enumerate}[label=\textup{(\roman*)},itemsep=1pt]
    \item if $K(S^x) \ge 4$, every shifted stutter qualifies, for
    instance the one with the doubles grouped by color and a group
    other than $z$'s last;
    \item if $K(S^x) = 3$ and $k_z \ge 4$, let $c, d$ be the other
    two colors (in either order) and take
    $w = c^{k_c} z^{k_z-2} d^{k_d}$.  If $k_z = 2$, let $c$ be a
    color with $k_c \ge 4$ (it exists since
    $S^x \ne \sig{2,2,2}$), let $d$ be the third color, and take
    $w = cc \cdot d^{k_d} \cdot c^{k_c-2}$.  The cyclic sequences
    are $(z, c\ldots, z\ldots, d\ldots)$ and $(z, c, d\ldots,
    c\ldots)$, each with exactly $4$ changes.
  \end{enumerate}
  \item \emph{$K(S^x) = 2$:} then $k_x = 1$ and $S = \sig{2p,2q,1}$
  up to order.  Since $S$ is not $\sig{2j,2,1}$, we have $p, q \ge 2$.
  For $Z = \{z,y\}$ put $v'_z = z^{k_z-1} y^{k_y} z$.  Then $C_z$ is
  a $4$-cycle and $\omega = y \ne x$.
\end{itemize}

The remaining case, $S^x = \sig{2,2,2}$, in which every shifted
stutter has exactly $3$ changes --- that is, $S$ a permutation of $\sig{3,2,2}$ or of $\sig{2,2,2,1}$ --- is handled in
\S\ref{one-odd:conclusion} by the certificate of Appendix~\ref{appendix:certificates}.

\begin{remark}[The three-color hypothesis is necessary]%
\label{one-odd:anchored-sharp}
Uniqueness can fail on two colors: $01100110$ and $10011001$ are both
anchored, and they share the skeleton
$\{0,1\}|\{0,1\}|\{0,1\}|\{0,1\}$, hence lie in the same
hypercube.  In the choice rule the shifted stutters use every color
of $S^x$, so there the hypothesis of the lemma below holds.
\end{remark}

\begin{lemma}[Anchored words]\label{one-odd:anchored}
An anchored word $v$ of length $2m$ that uses at least three
distinct colors is the only anchored word in its hypercube.
Consequently, in the first two cases of the choice rule the hypercubes
$C_z$, $z \in Z$, are pairwise distinct.
\end{lemma}

\begin{proof}
Write $v = a_0 a_1a_1 \cdots a_{m-1}a_{m-1} a_m$, so that
domino $i$ is $\{a_i, a_{i+1}\}$.  Let $v' = b_0 b_1b_1 \cdots b_m$ be
anchored with the same skeleton, so $\{b_i, b_{i+1}\} =
\{a_i, a_{i+1}\}$ for all $i$.  If $b_0 = a_0$, then $b_i = a_i$ for
all $i$, by induction.  Otherwise $b_0 = a_1 \ne a_0 = b_1$.  Suppose
$b_i = a_{i+1}$, $b_{i+1} = a_i$ and $a_i \ne a_{i+1}$.  Then
$\{a_i, b_{i+2}\} = \{a_{i+1}, a_{i+2}\}$ forces $a_{i+2} = a_i$ and
$b_{i+2} = a_{i+1}$, and so $a_{i+1} \ne a_{i+2}$: the hypothesis
holds at $i+1$.  By induction $a_{i+2} = a_i$ for all $i$, so $v$
uses only the colors $a_0, a_1$, a contradiction.  In the first two
cases of the choice rule, $v'_z$ has content $S^x$ and uses at
least three colors; different $z$ give different first colors, hence
different anchored words, hence different hypercubes.
\end{proof}

\begin{proposition}[Routing]\label{one-odd:routing}
Let $S$ be as in Theorem~\ref{thm:one-odd}, not a permutation of $\sig{2j,2,1}$, with $S^x \ne \sig{2,2,2}$, and choose the
$v'_z$ by the rule above.  Then, for every choice of Hamiltonian
cycles $K_z$ of the blocks $T_z$, the graph $N(S^x)$ has a
Hamiltonian cycle $\Theta$ as required in
Proposition~\ref{one-odd:assembly}(b).
\end{proposition}

\begin{proof}
Let $\Theta_0$ be a Hamiltonian cycle of $N(S^x)$ constructed as
in Section~\ref{sec:all-even}, with any admissible choices, let $H_C$
be its hypercube cycles and let $\mathcal{G}$ be its stitches.  An edge of
$C_z$ at $v'_z$ lies on no stitch of $\mathcal{G}$
(Lemma~\ref{cube:lemmaN}), so it lies on $\Theta_0$ if and only if it
lies on $H_{C_z}$.

\emph{Cases (i) and (ii).}  Here $\dim C_z \ge 4$.  By
Lemma~\ref{one-odd:shifted}(iii), $K_z$ has a non-exit edge $f_z$ at
$v_z$; its partner $g_z$ is an edge of $C_z$ in a non-last direction,
so it is not a rung.  By property \textup{(P)}
(Proposition~\ref{all-even:propertyP}), the stitches of $\mathcal{G}$
meet $C_z$ only in rungs.  Replace $H_{C_z}$ by a boustrophedon cycle
in direction $l(C_z)$ through $g_z$
(Lemma~\ref{cube:boustrophedon}).  This is one replacement per hypercube,
since the $C_z$ are pairwise distinct
(Lemma~\ref{one-odd:anchored}).  The new $H_{C_z}$ contains all rungs,
hence all unit edges of $\mathcal{G}$ in $C_z$, so
Lemma~\ref{prelim:tree-gluing} still applies, and the resulting
Hamiltonian cycle $\Theta$ contains every $g_z$.

\emph{Case (iii).}  Take $\Theta = \Theta_0$.  The hypercube $C_z$ is a
$4$-cycle, so $H_{C_z} = C_z$ contains both edges of $C_z$ at
$v'_z$.  ($C_z$ may also carry the non-rung edge of a descent or root
stitch, at most one by Lemma~\ref{all-even:lemmaI}; by
Lemma~\ref{cube:lemmaN} it is not at $v'_z$.)  Here
$\deg_{T_z}(v_z) = 1 + 1$, so $K_z$ uses the unique non-exit edge
$f_z$ at $v_z$, and its partner $g_z$ lies on $\Theta_0$.
\end{proof}

\subsection{Conclusion of the proof}\label{one-odd:conclusion}

If $S$ is a permutation of $\sig{2j,2,1}$, apply
Proposition~\ref{family:prop} for $j \ge 2$, or
Appendix~\ref{appendix:certificates} for $j = 1$.  Otherwise let
$K_z$ be arbitrary Hamiltonian cycles of the blocks $T_z$, which
exist by Proposition~\ref{one-odd:assembly}(a).

\begin{itemize}[itemsep=2pt]
  \item If $S^x = \sig{2,2,2}$ (so $S$ is a permutation of $\sig{3,2,2}$ or of $\sig{2,2,2,1}$), rename the colors of $S^x$
  as $0, 1, 2$ and take
  $v'_z = z\,(z+1)(z+1)\,(z+2)(z+2)\,z$ for $z \in Z$, with indices
  modulo $3$.
  \begin{itemize}[itemsep=1pt]
    \item The hypercube $C_z$ has dimension $3$, and its non-last
    directions are $0$ and $1$.  So the non-exit edges at $v_z$ are
    $s_0$ and $s_2$, whether or not $\omega = z+2$ equals $x$.
    \item Their partners are the position-$0$ and position-$2$ edges
    at $v'_z$.
    \item The Hamiltonian cycle $\Theta$ of $N(2,2,2)$ in
    Appendix~\ref{appendix:certificates} contains both of these edges
    at each $v'_z$.  So whichever non-exit edge $K_z$ uses
    (Lemma~\ref{one-odd:shifted}(iii)), its partner lies on $\Theta$.
  \end{itemize}
  \item Otherwise take $\Theta$ from
  Proposition~\ref{one-odd:routing}.
\end{itemize}

In both cases Proposition~\ref{one-odd:assembly}(b) gives a
Hamiltonian cycle of $N(S)$, proving Theorem~\ref{thm:one-odd}.

\section{\texorpdfstring{The family $\sig{2j,2,1}$}{The family (2j,2,1)}}
\label{sec:family}

The one-odd theorem excludes the family $\sig{2j,2,1}$ from its
assembly (Proposition~\ref{one-odd:assembly}), because one of its
blocks has no Hamiltonian cycle.  The block ending in the color of odd
multiplicity is all even, and the block ending in $0$ is a two-odd
graph with a Hamiltonian cycle, but the block
$T_1 = G(2j,1,1)\cdot 1$ is path-only
(Proposition~\ref{prelim:path-only-2k11}).  The block is
absorbed by a grid construction instead, with the gluing running
through edges that every Hamiltonian cycle of $T_0$ is forced to use.

Let $S = \sig{2j,2,1}$ on the colors $0,1,2$ with multiplicities
$2j,2,1$ (so $x = 2$), and $n = 2j+3$.  Verhoeff
\cite[\S5.3]{verhoeff2017} sketches an inductive construction of this
family (a non-inductive one is attributed there to van Duijnhoven
\cite{vanduijnhoven2013}), glued through Hamiltonian paths with
prescribed ends, leaving the final gluing step implicit.  The
construction below is complete and self-contained, and glues through
forced edges; for $j = 1$,
Appendix~\ref{appendix:certificates} lists a Hamiltonian cycle of
$N(2,2,1)$.  Let $j \ge 2$.  By Lemma~\ref{one-odd:blocks}, $N(S)$
splits into three blocks:
\begin{itemize}[itemsep=0pt]
  \item $T_2 = N(2j,2) \cdot 2$, which is Hamiltonian by
  Theorem~\ref{thm:all-even};
  \item $T_0 = G(2j-1,2,1) \cdot 0$, which is Hamiltonian by the
  inherited theorem (Theorem~\ref{inherited:stachowiak}; its odd
  multiplicities are $2j-1 \ge 3$ and $1$);
  \item $T_1 = G(2j,1,1) \cdot 1$, to which the inherited theorem
  does not apply.
\end{itemize}

\subsection{\texorpdfstring{The grid of $G(2j,1,1)$}{The grid of G(2j,1,1)}}

A word of signature $(2j,1,1)$ has length $N = 2j+2$.  Record it as
the cell $(p,q)$, where $p$ is the position of the color $2$ and $q$
the position of the color $1$.  Its neighbors are the cells
$(p \pm 1, q)$ and $(p, q \pm 1)$ of $\{0,\ldots,N-1\}^2$ off the
diagonal $p = q$ (the color $2$ or $1$ passes a $0$), and, when
$\lvert p-q \rvert = 1$, the cell $(q,p)$ (the colors $2$ and $1$
exchange).  So the upper triangle $\Delta^+ = \{p < q\}$ and the
lower triangle $\Delta^- = \{p > q\}$ induce grid graphs, and the
transposition $(p,q) \mapsto (q,p)$ is an automorphism that swaps
them.  The only edges between $\Delta^+$ and $\Delta^-$ are the
$2j+1$ edges $d_i d'_i$, where $d_i = (i,i+1)$ and
$d'_i = (i+1,i)$, $0 \le i \le 2j$.  Row $p$ of $\Delta^+$ is
$(p,p+1), \ldots, (p,N-1)$; it starts at $d_p$.

\begin{lemma}\label{family:gamma}
For $j \ge 2$, let $\psi_0 = d_{2j-1} = 0^{2j-1}210$ and
$\psi'_0 = d'_{2j-1} = 0^{2j-1}120$.  Then
$G(2j,1,1) - \{\psi_0, \psi'_0\}$ has a Hamiltonian cycle $\Gamma$
that contains the row-$0$ edge $(0,2j-1)(0,2j)$, that is,
$\{20^{2j-2}100,\ 20^{2j-1}10\}$.
\end{lemma}

\begin{proof}
Let $\Delta^+_* = \Delta^+ - \{\psi_0\}$.  Its rows $0, \ldots,
2j-2$ are complete; row $2j-1$ is the single cell $(2j-1, N-1)$ and
row $2j$ is the single cell $(2j, N-1) = d_{2j}$.  We cover
$\Delta^+_*$ by vertex-disjoint paths in two ways:
\begin{itemize}[itemsep=2pt]
  \item $Y_1$: for $0 \le i \le j-2$, the U-turn that starts at
  $d_{2i}$, runs along row $2i$ to column $N-1$, steps down, and runs
  back along row $2i+1$ to $d_{2i+1}$; and the path that starts at
  $d_{2j-2}$, runs along row $2j-2$ to column $N-1$, and then goes
  down through $(2j-1, N-1)$ to $d_{2j}$.
  \item $Y_2$: the path that starts at $d_0$, runs along row $0$ to
  column $N-1$, and goes down column $N-1$ to $d_{2j}$; and for
  $1 \le i \le j-1$, the U-turn that starts at $d_{2i-1}$, runs along
  row $2i-1$ to column $N-2$, steps down, and runs back along row
  $2i$ to $d_{2i}$.
\end{itemize}
The cover $Y_1$ covers rows $0, \ldots, 2j-3$ by its U-turns and rows
$2j-2, 2j-1, 2j$ by its last path; the cover $Y_2$ covers row $0$ and
column $N-1$ by its first path and the rest of rows $1, \ldots, 2j-2$
by its U-turns.  In both covers the path ends are exactly the cells
$d_i$, $i \in I = \{0, 1, \ldots, 2j-2, 2j\}$, each once.

Put the paths of $Y_1$ in $\Delta^+$, the transposes of the paths of
$Y_2$ in $\Delta^-$, and add the edges $d_i d'_i$ for
$i \in I$.  Every vertex of $G(2j,1,1) - \{\psi_0, \psi'_0\}$ now has
degree $2$.  Contract each path to an edge between the indices of
its two ends:
\begin{itemize}[itemsep=0pt]
  \item $Y_1$ gives the pairs $\{0,1\}, \{2,3\}, \ldots,
  \{2j-4, 2j-3\}, \{2j-2, 2j\}$;
  \item $Y_2$ gives the pairs $\{0, 2j\}, \{1,2\}, \{3,4\}, \ldots,
  \{2j-3, 2j-2\}$.
\end{itemize}
Their union is the single cycle $0, 1, 2, \ldots, 2j-2, 2j, 0$, so
$\Gamma$ is one cycle.  Row $0$ lies on the first U-turn of $Y_1$, so
$\Gamma$ contains $(0,2j-1)(0,2j)$.
\end{proof}

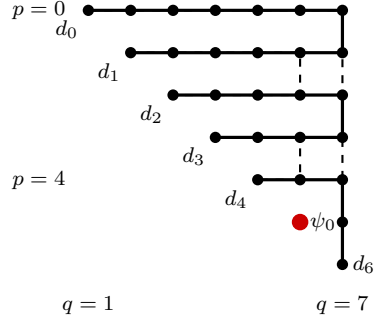
\begin{figure}[ht]
\centering
\begin{tikzpicture}[scale=0.56,
  dot/.style={circle,fill=black,inner sep=1.5pt}]
  \foreach \p in {0,...,7}{
    \foreach \q in {0,...,7}{
      \ifnum\q>\p \node[dot] at (\q,-\p) {}; \fi
    }
  }
  \draw[thick,dashed] (1,0)--(7,0)--(7,-6);
  \draw[thick,dashed] (2,-1)--(6,-1)--(6,-2)--(3,-2);
  \draw[thick,dashed] (4,-3)--(6,-3)--(6,-4)--(5,-4);
  \draw[very thick] (1,0)--(7,0)--(7,-1)--(2,-1);
  \draw[very thick] (3,-2)--(7,-2)--(7,-3)--(4,-3);
  \draw[very thick] (5,-4)--(7,-4)--(7,-6);
  \node[circle,fill=red!80!black,inner sep=2.2pt] at (6,-5) {};
  \node[below left,font=\scriptsize] at (1,0) {$d_0$};
  \node[below left,font=\scriptsize] at (2,-1) {$d_1$};
  \node[below left,font=\scriptsize] at (3,-2) {$d_2$};
  \node[below left,font=\scriptsize] at (4,-3) {$d_3$};
  \node[below left,font=\scriptsize] at (5,-4) {$d_4$};
  \node[right,font=\scriptsize] at (7,-6) {$d_6$};
  \node[right,font=\scriptsize] at (6,-5) {$\psi_0$};
  \node[left,font=\scriptsize] at (0.7,0) {$p=0$};
  \node[left,font=\scriptsize] at (0.7,-4) {$p=4$};
  \node[below,font=\scriptsize] at (1,-6.5) {$q=1$};
  \node[below,font=\scriptsize] at (7,-6.5) {$q=7$};
\end{tikzpicture}
\caption{The two covers of $\Delta^+_*$ for $j = 3$ ($N = 8$): cells
are dots, solid is $Y_1$, dashed is $Y_2$, and the red cell is the
deleted $\psi_0$.  The ends $d_i$, $i \in I$, are labeled.  In
$\Gamma$ the paths of $Y_1$ are placed in $\Delta^+$ and the transposes
of the paths of $Y_2$ in $\Delta^-$.}
\label{family:figure}
\end{figure}

\begin{proposition}\label{family:prop}
For every $j \ge 2$, $N(2j,2,1)$ has a Hamiltonian cycle.
\end{proposition}

\begin{proof}
Let $K_0$ be any Hamiltonian cycle of $G(2j-1,2,1)$
(inherited Theorem~\ref{inherited:stachowiak}), let $\Theta$ be any Hamiltonian
cycle of $N(2j,2)$ (Theorem~\ref{thm:all-even}), and let $\Gamma$ be
the cycle of Lemma~\ref{family:gamma}.  Consider the words
\begin{align*}
  \psi_1 &= 0^{2j-1}211, & \psi_2 &= 0^{2j-1}112, &
  \psi_3 &= 20^{2j-1}11 & &\text{of } G(2j-1,2,1),\\
  \psi_4 &= 0^{2j-1}110 &&&&&\text{of } N(2j,2).
\end{align*}
Each has exactly two neighbors in its graph: $\psi_1$ has
$s_{2j-2}(\psi_1)$ and $s_{2j-1}(\psi_1)$; $\psi_2$ has
$s_{2j-2}(\psi_2)$ and $s_{2j}(\psi_2)$; $\psi_3$ has
$s_0(\psi_3)$ and
$s_{2j-1}(\psi_3)$; and $\psi_4$ has $s_{2j-2}(\psi_4) =
0^{2j-2}1010$ and $s_{2j}(\psi_4) = 0^{2j-1}101$, both
non-stutters
because their domino $j-1$ is mixed.

A Hamiltonian cycle uses both edges at a vertex of degree $2$.  So
$K_0$ contains the three distinct edges
\[
  E_1 = \{\psi_1, \psi'_1\},\quad \psi'_1 = 0^{2j-1}121;\qquad
  E_2 = \{\psi_2, \psi'_2\},\quad \psi'_2 = 0^{2j-2}1012;
\]
\[
  E_3 = \{\psi_3, \psi'_3\},\quad \psi'_3 = 20^{2j-2}101,
\]
and $\Theta$ contains $E_4 = \{\psi_4, \psi'_4\}$ with
$\psi'_4 = 0^{2j-2}1010$.  With $\rho = s_{n-2}$, and writing
$(p,q) \cdot 1$ for the word of the cell $(p,q)$ followed by $1$,
\[
  \rho(\psi_1 \cdot 0) = \psi_0 \cdot 1,\qquad
  \rho(\psi'_1 \cdot 0) = \psi'_0 \cdot 1;
\]
\[
  \rho(\psi_2 \cdot 0) = \psi_4 \cdot 2,\qquad
  \rho(\psi'_2 \cdot 0) = \psi'_4 \cdot 2;
\]
\[
  \rho(\psi_3 \cdot 0) = (0,2j) \cdot 1,\qquad
  \rho(\psi'_3 \cdot 0) = (0,2j-1) \cdot 1.
\]
The units are $K_0 \cdot 0$, $\Theta \cdot 2$, $\Gamma \cdot 1$ and
the digon $\{\psi_0 \cdot 1, \psi'_0 \cdot 1\}$ (its two words
are joined by $s_{2j-1}$).  They are vertex-disjoint and cover
$T_0 \cup T_2 \cup T_1 = N(S)$ (Lemma~\ref{one-odd:blocks}).  In each
row of the table below, $w$ and $i$ satisfy the hypotheses of
Lemma~\ref{one-odd:parallel}: $w_i \ne w_{i+1}$ and
$w_{n-2} \ne w_{n-1}$ (the colors are listed), and
$i \le n-4 = 2j-1$.  So $f = \{w, s_i(w)\}$ and $\rho(f)$ form a
parallel pair.  In row $k$ we have $w = \psi_k \cdot 0$ and
$s_i(w) = \psi'_k \cdot 0$, so $f = E_k \cdot 0$; $\rho(f)$ is read
off the list of $\rho$-images above.

\begin{center}
\footnotesize
\setlength{\tabcolsep}{2pt}
\begin{tabular}{|c|l|c|l|l|l|l|}
\hline
 & $w$ & $i$ & colors at $i, i+1$; at $n-2, n-1$ & $f$, on $K_0 \cdot 0$ & $\rho(f)$ & unit of $\rho(f)$\\
\hline
1 & $\psi_1 \cdot 0 = 0^{2j-1}2110$ & $2j-1$ & $2, 1$; $1, 0$ & $E_1 \cdot 0$ & $\{\psi_0, \psi'_0\} \cdot 1$ & the digon\\
2 & $\psi_2 \cdot 0 = 0^{2j-1}1120$ & $2j-2$ & $0, 1$; $2, 0$ & $E_2 \cdot 0$ & $E_4 \cdot 2$ & $\Theta \cdot 2$\\
3 & $\psi_3 \cdot 0 = 20^{2j-1}110$ & $2j-1$ & $0, 1$; $1, 0$ & $E_3 \cdot 0$ & $\{(0,2j), (0,2j-1)\} \cdot 1$ & $\Gamma \cdot 1$\\
\hline
\end{tabular}
\end{center}

In each row the four words are non-stutters: $T_0$ and $T_1$
have none (Lemma~\ref{one-odd:blocks}), and $\psi_4 \cdot 2$ and
$\psi'_4 \cdot 2$ are non-stutters.  The two unit edges lie in
different units, so the gadget condition of
Lemma~\ref{one-odd:parallel} holds.  The three $4$-cycles form a star
centered at $K_0 \cdot 0$, whose unit edges $E_1 \cdot 0$,
$E_2 \cdot 0$, $E_3 \cdot 0$ are distinct and lie on $K_0 \cdot 0$.
The other unit edges lie on the digon, which carries only this one,
on $\Theta \cdot 2$ (the edge $E_4$), and on $\Gamma \cdot 1$
(Lemma~\ref{family:gamma}).  Lemma~\ref{prelim:tree-gluing} gives
one Hamiltonian cycle.
\end{proof}

\section{The main theorem}
\label{sec:assembly}

The two preceding sections construct the cycle at zero and at one odd
multiplicity; the remaining signatures, with two or more odd
multiplicities, are covered by the inherited theorem.  The assembly is
therefore a three-way dispatch.

\begin{theorem}[Reformulated Lehmer conjecture]%
\label{assembly:theorem}
Let $S$ be a signature of arity at least $2$ that is neither a binary
signature with an odd multiplicity nor a permutation of $\sig{2k,1,1}$.  Then $N(S)$ admits a Hamiltonian cycle.
\end{theorem}

\begin{proof}
Dispatch on the number of odd multiplicities of $S$.
\begin{enumerate}[label=\textup{(\arabic*)},itemsep=2pt]
  \item None: all multiplicities are even, and
  Theorem~\ref{thm:all-even} supplies a Hamiltonian cycle of $N(S)$.
  This includes the all-even binary signatures, so no clause of the
  binary classification is needed for existence here.
  \item One: the arity is at least $3$ by the first hypothesis (a
  binary signature with one odd multiplicity is excluded), and
  Theorem~\ref{thm:one-odd} supplies the cycle.
  \item At least two: such a signature has no stutter, so
  $N(S) = G(S)$.  Its arity is at least $3$: arity $2$ would mean two
  odd multiplicities, the excluded binary family.  It is not a
  permutation of $\sig{2j,1,1}$ by the second hypothesis, so the
  inherited theorem (Theorem~\ref{inherited:stachowiak}) supplies a
  Hamiltonian cycle of $G(S) = N(S)$.
\end{enumerate}
\end{proof}

\begin{remark}[Exactness of the exceptions]\label{assembly:exact}
At each excluded family, $N(S)$ admits a Hamiltonian path but no
Hamiltonian cycle: for the binary signatures with an odd multiplicity
this is Theorem~\ref{inherited:binary-theorem}, and for the
permutations of $\sig{2k,1,1}$ it is
Proposition~\ref{prelim:path-only-2k11}.  So the exception set of
Theorem~\ref{assembly:theorem} is minimal: outside it a cycle exists,
and at each member a path exists and no cycle does.
\end{remark}

\begin{remark}[Lehmer's conjecture]\label{assembly:lehmer}
By Verhoeff's reduction (Theorem~\ref{prelim:restriction}), a
Hamiltonian cycle or path of $N(S)$ yields a Lehmer path of $G(S)$.
Combining Theorem~\ref{assembly:theorem} with the paths of
Remark~\ref{assembly:exact} therefore gives Lehmer's word
conjecture in full: every neighbor-swap graph $G(S)$ admits a walk
that visits every word, with spur tips exactly at the
stutters.
\end{remark}

\begin{remark}[Well-foundedness, and the absence of mutual
induction]\label{assembly:wellfounded}
The dispatch above is not recursive, and Theorems~\ref{thm:all-even}
and~\ref{thm:one-odd} are not mutually inductive.
Theorem~\ref{thm:all-even} is a direct construction: given the
matching of Lemma~\ref{tools:lemmaM} --- whose proof is an induction
on the number $m = n/2$ of dominoes, and nothing else --- the stitches
are glued along an explicit spanning tree, with no recursion on the
length.  Theorem~\ref{thm:one-odd} reduces a one-odd signature to the
strictly shorter all-even signature $S^x$
(Theorem~\ref{thm:all-even}) and to the inherited two-odd theorem for
its blocks; the family $\sig{2j,2,1}$ likewise reduces to
Theorem~\ref{thm:all-even} on $\sig{2j,2}$ and to the inherited
theorem on $\sig{2j-1,2,1}$, both strictly shorter.  The dependency
is therefore acyclic --- the one-odd proof uses the all-even proof
and never the converse --- every chain of appeals is finite, and the
inherited theorem introduces no recursion from this article's side.
\end{remark}

\begin{remark}[Machine verification]\label{assembly:machine}
The assembled claim has its own standing gate, independent of the
producers' gates: it re-derives $N(S)$ from the definition, checks
that the dispatch of the proof is exhaustive and disjoint and that its
exception predicate agrees with the classification of
Section~\ref{prelim:trichotomy-section}, and then runs the route end
to end at every signature in the gate's window --- the all-even
construction, the one-odd construction (whose two-odd blocks are
built by the Stachowiak implementation), and, at two or more odd
multiplicities, Stachowiak's construction itself --- verifying every
returned cycle vertex by vertex and edge by edge against the gate's
own graph.  Appendix~\ref{appendix:machine-verification} states the
exact claims and parameters.
\end{remark}

\section{Conclusion}
\label{sec:conclusion}

This article proves the reformulated Lehmer conjecture
(Theorem~\ref{assembly:theorem}): outside the two exceptional families
--- binary signatures with an odd multiplicity, and the permutations
of $\sig{2k,1,1}$ --- the non-stutter graph $N(S)$ of a signature of
arity at least $2$ admits a Hamiltonian cycle.  Equivalently, it
solves Exercise~53 of \S7.2.1.3 of Knuth's \emph{Art of Computer
Programming} \cite{knuth2011}, the form in which Lehmer's question is
posed there.  At the exceptional
families a Hamiltonian path continues to exist and no cycle does, so
the exception set is exact, and Verhoeff's reduction turns the cycles
and paths into the Lehmer paths of the original conjecture
(Remark~\ref{assembly:lehmer}).

The proof is built on the partition of the words into
hypercubes.  The even (intra-domino) swaps generate the parabolic
subgroup whose orbits are the hypercubes of
Lemma~\ref{cube:lemma1}; in the all-even case these orbits are
hypercubes of dimension at least $2$, and the proof glues their
boustrophedon cycles along a spanning tree of stitches --- the rung
stitches inside each family built from the matching of
Lemma~\ref{tools:lemmaM}, and at most one descent or root stitch
leaving each family (Lemma~\ref{all-even:lemmaI}).  The one-odd
theorem glues the blocks of its trailing-color decomposition along
parallel pairs whose prescribed edge is read off the all-even cycle,
using the shifted stutters to force one edge inside every block
without prescribing anything inside the inherited two-odd blocks.
The family $\sig{2j,2,1}$ is absorbed through degree-$2$ forced
edges, and signatures with two or more odd multiplicities are the
inherited Stachowiak theorem.

The article relies on exactly two inherited Hamiltonicity results,
both stated in Section~\ref{sec:inherited}: Stachowiak's multi-odd
theorem, consumed as an arbitrary-cycle black box, and Verhoeff's
binary classification, consumed only for the exactness of the
exception set.  No availability hypothesis, counting condition, or
finite search remains on the route: the only finite ingredients are
the two explicit cycles of Appendix~\ref{appendix:certificates}, of
$28$ and $84$ words.

The proof was developed with AI assistance in a way that the
introduction records in condensed form
(\S\ref{intro:provenance}).  Two supporting artifacts accompany the
article.  Every construction has an executable form in the Python
package \texttt{nswap}, machine-checked against brute-force
neighbor-swap graphs at small parameters
(Appendix~\ref{appendix:machine-verification}); and the hypercube proof
and the Stachowiak theorem are formalized in Lean~4 over Mathlib,
with a machine-checked bridge supplying the two-odd blocks
(Appendix~\ref{appendix:formalization}).  The Lean development is documented
statement by statement, so that it can be ported and, in the parts
that are stable and self-contained, offered to Mathlib.

\subsection{Future work}

\begin{itemize}[itemsep=2pt]
  \item \emph{A self-contained two-odd proof.}  The inherited
  theorem is the only multi-odd input.  A self-contained proof of
  the exactly-two-odd case would reduce the inherited input to the
  three-or-more-odd case, and a self-contained treatment of that case
  would remove it entirely.
  \item \emph{The path side.}  The reformulated conjecture also has a
  path form at the two exceptional families, and Lehmer's original
  conjecture is about paths; the present article consumes the path
  clauses only for exactness.  A uniform treatment of the path
  deliverables remains open, though the literature covers the
  families.
  \item \emph{Mathlib.}  The hypercube layer, the Stachowiak package, and
  the two-odd bridge are candidates for contribution to Mathlib once
  the toolchain port and the naming and documentation pass are
  complete.
\end{itemize}

\section*{Acknowledgments}
\addcontentsline{toc}{section}{Acknowledgments}

I thank Anton Bakker for challenging me to pick up this problem again
last March, when Opus~4.7 appeared: ``Software is done, and I bet
mathematics as well,'' he said, and I replied that I did not believe it
one bit.  Two years earlier, AI had still seemed like a child at
research-level mathematics, so I was happily surprised by the mature
discussion that followed, and I pursued the problem.  I worked with
various Claude models --- Opus~4.7, Opus~4.8, Fable~5 (and
Fable~5.1), and Opus~5 --- but progress
was slow; after some six months the proof (then almost 200 pages of
dense mathematics) and its Lean formalization (more than one hundred
thousand lines) came together.  Then Opus~5.5 appeared, and Anton
asked it to come up with a simple proof from scratch, which it did
within a couple of hours.
So thank you, Anton, for hanging in with me.  I was also helped by
several people at the Supercomputing Center (SCC) of Eindhoven
University of Technology, and
I thank them for allowing me to use Spike-1.

The hypercube proof at the heart of this article was found and developed by
Claude Opus 5.5 in a multi-agent proof effort.  I also thank Harmonic for the Aristotle system that
produced the Lean formalization, and the participants of the earlier
development route of this project.  The conjecture treated here is due to
D.~H.~Lehmer \cite{lehmer1965}; the reformulation in terms of stutter
words, the binary classification, and the open case's cycle
cover are due to T.~Verhoeff \cite{verhoeff2017}; the multi-odd
theorem is due to G.~Stachowiak \cite{stachowiak1992}; and
M.~Opperman \cite{opperman2024} pursued the merge-of-covers strategy
in his master's thesis.

\section*{Note on authorship and AI collaboration}
\addcontentsline{toc}{section}{Note on authorship and AI collaboration}
\label{front:ai-note}

The author is the lead author: he set the direction of the project,
made all mathematical and editorial decisions, supervised every
session, and takes responsibility for the contents.  The work was
carried out with AI systems, in three roles.

\begin{itemize}[itemsep=2pt]
  \item \textbf{Claude Opus 5.5} (Anthropic) found the two ideas at
  the heart of the proof --- the partition of the words into
  hypercubes (the orbits of the intra-domino swaps, with the stutters
  as their fixed points) and the matching of odd swaps that spans the
  hypercubes --- in a study phase
  that surveyed untried representations, and then developed and
  audited the proof in successive multi-agent rounds.  This article's
  proof text re-derives that proof in the article's notation, with
  independent machine verification by the project.
  \item \textbf{Aristotle} (Harmonic) formalized the hypercube proof and
  the Stachowiak construction in Lean~4 over Mathlib, as separate
  packages whose export is built and checked in the accompanying Lean
  development (Appendix~\ref{appendix:formalization}).
  \item \textbf{DeepSeek V41 Flash} (run on Spike-1, the supercomputer
  of Eindhoven University of Technology) performed the project-side
  integration: the bridges
  between the packages, the independent \texttt{nswap} implementations
  and gates of Appendix~\ref{appendix:machine-verification}, and the
  preparation of this article.
\end{itemize}

\section*{The ornamental braid}
\addcontentsline{toc}{section}{The ornamental braid}
\label{back:braid}

The decorated edition of this article carries an ornamental braid drawn
as a banner across the top of every page, spanning the text width: a
Lehmer cycle on
the words with signature $(4,2,3)$, drawn with one strand per color
and one crossing per neighbor swap, the strands colored blue ($0$),
orange ($1$) and green ($2$).  Here $2$ is the color of odd
multiplicity: it occurs three times, while $0$ occurs four and $1$ two times.  The cycle
is the Hamiltonian cycle of the non-stutter words from the
one-odd theorem (Theorem~\ref{thm:one-odd}) --- the $1248$
words, with the $12$ stutters attached as out-and-back kinks, so
the drawn braid has $1272$ crossings --- assembled by the block
construction of Section~\ref{sec:one-odd}.  Blocks are read off the
last color, and the cycle runs through four arcs in the order $0$,
$2$, $1$, $2$, of lengths $560$, $236$, $280$ and $172$ crossings:
the $408$ words ending in $2$ --- the words of the all-even signature
$(4,2,2)$ with a final $2$ appended --- carry the all-even cycle (Theorem~\ref{thm:all-even}) in two arcs, and the $560$
words ending in $0$ and the $280$ ending in $1$ carry cycles
from the inherited multi-odd theorem
(Theorem~\ref{inherited:stachowiak}), each in one arc.  The seams
between the all-even and the two-odd arcs are the four swaps of the
last two colors at the shifted stutters
(Lemma~\ref{one-odd:shifted}), here the words $01100220$ and $10022001$
--- each with a final $2$, or one swap of such a word.  The structure
is made visible.  Inside the all-even block the background is a
checkerboard, one column per domino, whose phase flips exactly at the
odd
swaps, so each hypercube visit sees one fixed checkerboard and consecutive
visits see opposite ones; the two-odd blocks are tinted flat, warm for
the block ending in $0$ and cool for the one ending in $1$.  Each of
the twelve stutters --- an out-and-back kink at its word ---
carries a thin double rule at its tip, and each of the four seams a
mark of its own.  Every word sits on a regular cross section
where the strands run straight; the crossings halfway between are drawn
perpendicular, and the level length is fitted per page so that the
braid spans the text width exactly.  The decoration is ornamental and
no part of the proof depends on it; the scripts and data that generate
it accompany the submission.

\addcontentsline{toc}{section}{References}
% References frozen from main.bbl (arXiv does not run BibTeX).

\appendix
\section{The two explicit cycles}
\label{appendix:certificates}

A cyclic sequence of words is a Hamiltonian cycle of $N(S)$
when consecutive words (cyclically) differ by a swap of two
adjacent distinct colors, the words are distinct non-stutters,
and there are as many of them as there are non-stutters.  The proof
uses the following two finite objects; each is small enough to check
by hand, and both are checked by the machine layer
(Appendix~\ref{appendix:machine-verification}).

\subsection{\texorpdfstring{$\sig{2,2,1}$ (used for $j = 1$)}{(2,2,1) (used for j = 1)}}
\label{appendix:certificates:a1}

Let the multiset be $\{0,0,1,1,2\}$ and let $N(2,2,1)$ have the $28$
non-stutters.  The following cyclic sequence of words is a Hamiltonian
cycle of $N(2,2,1)$:
\begin{center}
\ttfamily\footnotesize
\begin{tabular}{lllllll}
00121 & 01021 & 01012 & 01102 & 10102 & 10012 & 10021\\
10201 & 01201 & 02101 & 02110 & 20110 & 21010 & 21100\\
12100 & 11200 & 11020 & 10120 & 01120 & 01210 & 10210\\
12010 & 12001 & 21001 & 20101 & 20011 & 02011 & 00211
\end{tabular}
\end{center}
The last word $00211$ is adjacent to the first word $00121$.  The
cycle is used in Section~\ref{one-odd:conclusion} for $j = 1$.

\subsection{\texorpdfstring{$\sig{2,2,2}$}{(2,2,2)}}
\label{appendix:certificates:a2}

Let the multiset be $\{0,0,1,1,2,2\}$ and let $N(2,2,2)$ have the
$84$ non-stutters.  Let $r$ be the color rotation $c \mapsto c+1
\pmod 3$; it is an automorphism of $N(2,2,2)$ that fixes no
word, so its orbits on the non-stutters have size $3$.  Let
$\Pi$ be the $28$-word path
\begin{center}
\ttfamily\footnotesize
\begin{tabular}{lllllll}
001212 & 002112 & 002121 & 001221 & 010221 & 100221 & 100212\\
100122 & 010122 & 011022 & 101022 & 101202 & 011202 & 012102\\
102102 & 120102 & 210102 & 210120 & 201120 & 201102 & 021102\\
021120 & 012120 & 011220 & 101220 & 102120 & 120120 & 121020
\end{tabular}
\end{center}
Then $\Theta = \Pi, r(\Pi), r^2(\Pi)$, closed up, is a Hamiltonian
cycle of $N(2,2,2)$.  The hand check has four parts:
\begin{enumerate}[label=\textup{(\arabic*)},itemsep=1pt]
  \item Consecutive words of $\Pi$ differ by one swap of two adjacent
  distinct colors (and so do those of $r(\Pi)$ and $r^2(\Pi)$).
  \item The last word $121020$ is adjacent to
  $r(001212) = 112020$; applying $r$ and $r^2$ gives the other two
  junctions.
  \item $\Pi$ meets each of the $28$ orbits of $r$ exactly once, so
  $\Theta$ visits each of the $3 \cdot 28 = 84$ non-stutters once.
  \item At each of $v'_0 = 011220$, $v'_1 = 122001$ and
  $v'_2 = 200112$, the cycle $\Theta$ contains the position-$0$ and
  the position-$2$ edge, and so not the position-$4$ edge.
\end{enumerate}
Fact \textup{(4)} is the one used in Section~\ref{one-odd:conclusion}
for $S^x = \sig{2,2,2}$; by the $r$-invariance of $\Theta$ it
suffices to check it at $v'_0$.

\section{Machine verification}
\label{appendix:machine-verification}

Every construction in this article is implemented, in the same
generality as the corresponding proof, in a Python package that is
distributed with this submission (\texttt{nswap}), and machine-checked
against brute-force neighbor-swap graphs at small parameters.  Each
verifier re-derives the construction exactly as the text describes it,
then checks the result against the actual graph $N(S)$: the exact
non-stutter vertex set, the validity of every swap, and the presence
of every guaranteed edge claimed in the statements and proofs.  The
package builds the objects that the proofs prescribe, without search
or solvers.  The mathematical proofs stand on their own; the machine
layer is an independent check of them, and the package's public
functions carry pre- and postconditions citing the article labels
they implement.

The package is built from the same definitions as the article.  The
hypercube constructions use only the core vocabulary (signatures, words,
the stutter test, the graph and its edge relation); the two-odd blocks
are served by the faithful implementation of the inherited theorem;
and the assembly has its own end-to-end gate.  The parameters below
are those of the standing runs.

{\footnotesize
\begin{longtable}{p{0.30\textwidth}p{0.62\textwidth}}
\caption{The machine-verified units of the article and the parameters
of the standing runs.}
\label{appendix:machine-verification:table}\\
\hline
\textbf{Unit} & \textbf{Claims, and parameters of the standing run}\\
\hline
\endfirsthead
\hline
\textbf{Unit} & \textbf{Claims, and parameters (continued)}\\
\hline
\endhead
\hline
\endfoot
Theorem~\ref{thm:all-even} (the all-even cycle) &
(CA1) hypercube structure (Lemma~\ref{cube:lemma1}): on every even-size
signature with $n \le 8$, each skeleton class is a hypercube on its
mixed directions ($\#\text{hypercube} = 2^r$), even swaps stay inside a
hypercube and odd swaps leave it, and for all-even signatures every
non-stutter hypercube has $r \ge 2$.  (CA2) Lemma~\ref{tools:lemmaM} and
(INV): on the same window, the returned edges are odd edges, form a
matching, and, between hypercubes, a spanning tree of the hypercubes; the free
words end in their colors.  (CA3) families, stitches,
Lemma~\ref{all-even:lemmaI}, property~\textup{(P)},
Lemma~\ref{cube:lemmaN}, and the stitch spanning tree, on every
all-even signature with $n \le 8$ (chain-of-families, non-rung
directions, distinctness, pairwise distinct rungs per hypercube).
(CA4) Theorem~\ref{thm:all-even} end to end on every all-even
signature with $n \le 8$: the constructed cycle lists every
non-stutter exactly once, consecutive words adjacent.  Brute-force
hypercube re-derivations are gate-local; the extended window CA1--CA3 at
$n \le 10$ ($682$ even-size, $26$ all-even signatures) also
passes.\\[2pt]
Theorem~\ref{thm:one-odd} (the one-odd cycle) &
(CO1) Lemma~\ref{one-odd:blocks}: on every one-odd signature with
$n \le 9$, the blocks partition the non-stutters, only the
$\rho$-edges cross between them, and no $T_z$ word is a stutter.
(CO2) Lemmas~\ref{one-odd:parallel} and~\ref{one-odd:shifted}: for
the general route, $v'_z$ is anchored, the dimension rule holds and
$\dim C_z \ge 4$ (the arity-$2$ case gives $2$), the rungs are the
last direction, the parallel pairs are $4$-cycles of $N(S)$ with
non-rung partners, the degree formula holds, and the forced edge is
a non-exit edge at $v_z$.  (CO3) Lemma~\ref{one-odd:anchored},
Lemma~\ref{cube:lemmaN}, property~\textup{(P)}, and the routing:
the $C_z$ are pairwise distinct, the $v'_z$ hypercube edges are met by no
stitch and $C_z$ is met by no non-rung stitch, and $\Theta$ is a
Hamiltonian cycle of $N(S^x)$ containing every partner $g_z$.
(CO4) the certificates and Section~\ref{sec:family}: A.1 is a
Hamiltonian cycle of $N(2,2,1)$; A.2 is a Hamiltonian cycle of
$N(2,2,2)$ containing the position-$0$ and position-$2$ edges at
every $v'_z$; the grid construction is the cycle of
Lemma~\ref{family:gamma} containing the row-$0$ edge, and the
family construction is a Hamiltonian cycle of $N(2j,2,1)$, for
$j \le 6$.  (CO5) Theorem~\ref{thm:one-odd} end to end on every
one-odd signature with $n \le 9$, with the two-odd block cycles from
the Stachowiak implementation.\\[2pt]
The family $\sig{2j,2,1}$ &
(CO4): as above, the grid cycle for $j \le 6$ and the assembled
cycle of $N(2j,2,1)$ for $j \le 6$; the local degree facts and the
$\rho$-identities of Proposition~\ref{family:prop} are re-derived
gate-locally, and the block cycles $K_0$ and $\Theta$ are drawn
from the independent implementations (Stachowiak and the all-even
construction).\\[2pt]
Theorem~\ref{assembly:theorem} (the assembled conjecture) &
(A1): the dispatch is exhaustive and disjoint at every multiset with
$\ge 2$ colors, $n \le 16$, arity $\le 8$ ($778$ multisets); the
exception predicate (binary with an odd multiplicity, or a
permutation of $\sig{2k,1,1}$) equals the path-only trichotomy, each
producer's domain hypothesis holds, and the package's own predicate
agrees.  (A2) a gate-local dispatch of
Theorem~\ref{assembly:theorem} at every in-scope signature with
$n \le 12$, $\lvert\operatorname{Perm}\rvert \le 20\,000$ ($210$
signatures: $56$ binary both-even to $n = 16$, $154$ of arities
$3$--$7$, all arrangements at arity $3$ with $n \le 9$), each a
Hamiltonian cycle of the independently re-derived $N(S)$ ($531\,918$
vertices in all, largest $18\,840$; $81$ exceptions refused): $66$
all-even by the all-even construction, $40$ one-odd by the one-odd
construction (the two-odd blocks from the Stachowiak
implementation), and $104$ multi-odd by the Stachowiak construction
itself; no search anywhere.  Standing run: $4.1$~s.\\[2pt]
Theorem~\ref{inherited:stachowiak} (the inherited input) &
(ST1)--(ST4): the faithful implementation of Stachowiak's 1992
construction (Lemmas~1--12 and the main theorem, via the conventions
recorded with the implementation), every object checked vertex by
vertex and edge by edge against adjacency recomputed from the swap
rule alone.  (ST1) structural re-derivation of the
$\lvert Q \rvert = 2$ layer; (ST2) its end-to-end window ($230$
cycles $+$ $12$ paths, $2\,899\,056$ vertices); (ST3) the series
compositions (Lemmas~7--9): $507$ cycles, $1\,934\,460$ vertices;
(ST4)
Lemmas~10--12 and the public construction: Lemma-10 structure ($352$
piece isomorphisms, $1\,408$ degree-$2$ forcings, the glue
dichotomy re-derived at $689$ path boundaries), Lemma-10 end
to end ($47$ cycles, $644\,066$ vertices), the Lemma-11 anchor pick
re-derived, and the public construction over $289$ multi-odd
signatures, $1\,575\,778$ vertices --- all arrangements of arity
$3$--$4$ with $n \le 10$ and families out to $n \le 16$ --- plus
$23$ refusals (Lemma~12's family and off-domain) with exact-search
no-cycle certificates at $\sig{2,1,1}$ and $\sig{4,1,1}$.  One
official run: $274$~s.\\[2pt]
Theorem~\ref{inherited:binary-theorem} (the binary classification) &
(VB1)--(VB6): Verhoeff's recursion, implemented deterministically,
delivers the Hamiltonian path or cycle at every binary signature
$(k_0,k_1)$ with $k_0 + k_1 \le 14$, checked against the
independently enumerated graph, in all three parity cases, including
the two forced endpoints of the odd--odd path and the two
guaranteed edges of the even--even cycle; determinism across fresh
interpreters; and the solver-free claim, the odd--odd, odd--even and
even--even objects building with every solver entry point
embargoed.\\
\end{longtable}
}

The \texttt{nswap} package also carries the standard vocabulary and
the deterministic binary construction used by the checks above.  The
package is about $5\,400$ lines of Python, and the standing verifiers
and their tests add about $3\,800$ more.  It needs Python~3.11 or
later and the standard library only; the standing runs have been
checked under Python~3.11 and~3.14.  The programs of the earlier,
much longer development route are not part of this submission and do
not enter the table.

\section{Lean formalization}
\label{appendix:formalization}

The proof of this article is formalized in Lean~4 over Mathlib
(Lean~4 v4.28.0, with the Mathlib release tagged \texttt{v4.28.0},
commit \texttt{8f9d9cf}), as three layers.
\begin{enumerate}[label=\textup{(\arabic*)},itemsep=2pt]
  \item The \emph{hypercube package} (the Lean package is named
  \texttt{Cube}; $49$ modules) formalizes the
  all-even and one-odd halves of the hypercube proof:
  Lemma~\ref{cube:lemma1} (its (b) clause),
  Lemma~\ref{cube:boustrophedon}, Lemma~\ref{tools:lemmaM} with
  (INV), Lemma~\ref{all-even:lemmaI}, property~\textup{(P)},
  Lemma~\ref{cube:lemmaN}, Theorem~\ref{thm:all-even}, and the
  one-odd ladder --- Lemmas~\ref{one-odd:blocks},
  \ref{one-odd:parallel}, \ref{one-odd:shifted},
  \ref{one-odd:anchored}, \ref{family:gamma},
  Propositions~\ref{one-odd:assembly}, \ref{one-odd:routing},
  \ref{family:prop}, and Theorem~\ref{thm:one-odd} --- conditional
  on one named input: the exactly-two-odd cycle of
  Proposition~\ref{one-odd:assembly}(a).  The gluing machinery of
  Lemma~\ref{prelim:tree-gluing} is developed as a separate
  sublibrary.
  \item The \emph{Stachowiak package} ($31$ modules) formalizes
  Stachowiak's paper \cite{stachowiak1992} --- its Lemmas~1--12 and
  both corollaries, paper-faithful, with the two-odd corollary
  carrying the paper's arity bound --- conditional on the odd--odd
  binary path it cites.
  \item A third layer supplies the bridges.  The odd--odd binary
  path is proved from the binary induction, discharging the
  Stachowiak package's input; the resulting multi-odd theorem also
  discharges the hypercube package's named input; and the hypercube graph is
  transported across the vertex and edge identities to the article's
  vocabulary, assembling the headline along the hypercube route plus
  Stachowiak alone.  The same layer formalizes, in the article's
  vocabulary, the preliminaries of Section~\ref{sec:preliminaries}:
  the parity census (Lemma~\ref{prelim:stutter-parity}), the
  distance between stutters (Lemma~\ref{prelim:stutter-distance}, in
  Verhoeff's stronger form that it is a multiple of~$4$), and
  Verhoeff's reduction (Theorem~\ref{prelim:restriction}), for
  Hamiltonian paths and cycles.  The no-cycle half of
  Proposition~\ref{prelim:path-only-2k11} is Stachowiak's Lemma~12 in
  the Stachowiak package.
\end{enumerate}

Together the three layers are about $37\,700$ lines of Lean: about
$10\,800$ for the hypercube package, $8\,300$ for the Stachowiak
package, and $18\,700$ for the bridges, the preliminaries, and the part
of the project library that they depend on.

The machine-checked headline is the theorem
\texttt{reformulatedLehmer\_of\_cube}, a proof of the reformulated
conjecture of this article; the one-odd theorem is available with
its named input substituted, and the at-most-one-odd half as a named
interface.  The development compiles with no \texttt{sorry} and no
stray \texttt{axiom}, and \texttt{\#print axioms}
on the headline targets and on the preliminaries reports exactly
\texttt{[propext, Classical.choice, Quot.sound]}.  The two packages
are included with the development and machine-checked there.  The
bridge modules connect them to the article's vocabulary, and their
transport theorems (the identification of stutters, vertex sets, edge
relations, and cycle systems) are themselves proved in Lean, so
nothing is assumed where the packages meet.

Each package is accompanied by a blueprint: a statement-by-statement
narrative that links every mathematical statement to its Lean
declaration and records whether it is formalized, stated, or
pending.  The blueprint of the hypercube package follows the order of this
article's sections and thus doubles as a second, machine-checked
table of contents for the proof.  The development is pinned at the
toolchain version of the original formalization; a port to the
current stable toolchain, and with it the preparation of the stable
parts for contribution to Mathlib, is in progress and follows the
same blueprint structure.

\end{document}